\documentclass[reqno,16pt]{amsart}
\usepackage[dvipdfmx]{graphicx}
\usepackage[dvips]{color}
\usepackage{amsmath,amsfonts,amssymb,amsthm,amscd}
\usepackage{epic}
\usepackage{eepic}
\usepackage{longtable}
\usepackage{array}
\usepackage{comment}

\usepackage{amsmath}
\usepackage{graphicx}
\usepackage{amsfonts}
\usepackage{amsthm}
\usepackage{mathrsfs}
\usepackage{bigints}
\usepackage{amssymb}
\usepackage{booktabs}
\usepackage{xspace}
\usepackage{tikz}
\usepackage{tikz-cd}
\usepackage{enumerate}
\usepackage{hyperref}
\usepackage{upgreek}
\usepackage{verbatim}
\usepackage{mathtools}
\mathtoolsset{showonlyrefs}

\usepackage[font=small]{caption}
\usepackage{ytableau}
\usepackage[vcentermath, enableskew]{youngtab}
\usepackage{hyperref}
\hypersetup{colorlinks,linkcolor={red},citecolor={olive},urlcolor={red}}

\usepackage{tikz}

\usepackage{enumitem}

\newcommand{\R}{\mathbb{R}}

\newcommand{\N}{\mathbb{N}}

\newcommand{\J}{[J]}

\newcommand{\wmass}{L}

\newcommand{\load}{\varrho}

\newcommand{\ar}{\alpha} 
\newcommand{\ap}{A} 

\newcommand{\inta}{u} 
\newcommand{\iinta}{U}
\newcommand{\pat}{\ell} 
\newcommand{\pd}{\vartheta} 
\newcommand{\ser}{v} 
\newcommand{\tm}{Z} 
\newcommand{\ssp}{\mathcal{Z}} 
\newcommand{\fl}{\zeta} 
\newcommand{\flm}{z} 

\newcommand{\avg}{H} 
\newcommand{\mart}{Y} 
\newcommand{\othermart}{\mathcal{Y}}

\newcommand{\sr}{\mu} 

\newcommand{\M}{\mathbf{M}}

\theoremstyle{definition}

\theoremstyle{definition}

\theoremstyle{definition}
\newtheorem{defi}{Definition}[section]
\theoremstyle{plain}
\newtheorem{lem}{Lemma}[section]
\theoremstyle{plain}
\newtheorem{prop}{Proposition}[section]
\theoremstyle{plain}
\newtheorem{thm}{Theorem}[section]
\theoremstyle{plain}

\theoremstyle{definition}
\newtheorem{rem}{Remark}[section]
\theoremstyle{definition}

\theoremstyle{definition}
\newtheorem{assumption}{Assumption}
\newcommand\numberthis{\addtocounter{equation}{1}\tag{\theequation}}
\begin{document}

\title[Diffusion Approximation of PPS Queue with Reneging]{Fluid and Diffusion Approximation of a Multiclass Proportional Processor Sharing Queue with Reneging and General Distributions}

\author{Eva Loeser}
\address{Eva Loeser, Department of Statistics and Operations Research, University of North Carolina at Chapel Hill, 204 E Cameron Ave, Chapel Hill, NC, 27514, USA}
\email{\texttt{ehloeser@unc.edu}}

\begin{abstract}
In this paper, we resolve the long open problem of the diffusion approximation for a generally distributed processor sharing queue with reneging. In particular, we first extend the known fluid limit for a single class processor sharing queue with reneging to the multiclass proportional processor sharing queue with reneging.
We prove uniqueness of fluid model solutions and convergence of fluid scaled models to fluid model solutions.
Next, we show that, under diffusion-scaling, the models are tight, and that subsequential limits can be uniquely characterized by a certain system of SDEs driven by a Gaussian sheet on $\R_+^2$ with a nontrivial covariation structure.

\end{abstract}

\maketitle

\section{Introduction}

Processor–sharing (PS) queues are a canonical model of fair resource allocation in communications, computing, and service systems. The processor sharing queue has received much study. In \cite{gromollpuhawilliams}, Gromoll, Puha, and Williams established the fluid limit for the generally distributed, single-class processor sharing queue without reneging.
Later, in \cite{gromollzwart}, Gromoll, Robert, and Zwart established the fluid limit for a single-class, generally distributed processor sharing queue with impatience.
Gromoll and Kruk also obtained a heavy traffic limit for the single-class processor sharing queue that tracks impatience \cite{gromollkruk}.
However, they removed the reneging from the model, opting for ``soft deadlines," in order to use techniques from heavy traffic approximations.
In particular, a state space collapse argument is used to obtain a diffusion approximation for the measure-valued state descriptor by lifting the diffusion limit of the workload process into the space of finite radon measures. To make this argument tractable, the authors assume {soft deadlines}, meaning that jobs do not leave the system when their patience times expire. This assumption allows the diffusion limit of the workload process to coincide with that of a classical PS queue without reneging, described by a reflected Brownian motion. The case with {firm deadlines} (i.e., reneging) is left for future work and described as exhibiting `very different behavior.''
In this work, we use a more classical framework for central limit theorems for systems of this type, following the outline given in \cite{loeser2025diffusionlimitsmeasurevaluedqueueing}.
With this methodology, we will extend the results from \cite{gromollzwart} to a multi-class processor sharing queue with reneging and obtain the diffusion approximation with reneging (hard deadlines).

We consider a single server, multiclass proportional processor sharing (PPS) queue with reneging and general interarrival, service, and patience distributions. There are \(J\) classes of jobs, with PPS weights \(p_1,\dots,p_J>0\), \(\sum_{j=1}^J p_j=1\). Jobs of class $j$ arrive according to a renewal process of rate \(\alpha_j\); upon arrival the $i$th job of class $j$ draws i.i.d.\ service and patience times \((\ser^j_i,\pat^j_i)\) from a general distribution \(\vartheta_j\). A job of class \(j\) present at time \(t\) is served at rate \(p_j / \sum_{i=1}^J p_i Z_i(t)\), where \(Z_j(t)\) is the number of class $j$ jobs in the system at time $t$; it abandons when its residual patience hits zero. The state of the class-$j$ jobs in the queue is described by a {measure–valued} process that places a point mass at the residual service and patience times for each job in the $j$th queue:
\[
\ssp_j(t)=\sum \delta_{(\text{res.\ service},\ \text{res.\ patience})}^+,
\qquad j\in[J].
\]
We write \(\boldsymbol{\ssp}(t)=(\ssp_1(t),\dots,\ssp_J(t))\). Define the linear map on $\wmass:\R_+^J\rightarrow\R_+,$
\begin{equation}
\label{wmassdef}
\wmass(\boldsymbol{z})=\sum_{j=1}^J p_j z_j.
\end{equation}
We also define the total mass process for the $j$th queue:
$$\tm_j(t):= \langle 1, \ssp_j(t) \rangle, \hspace{5mm} t\geq 0,$$
and once again write $\boldsymbol{\tm}(t):=(\tm_1(t),...,\tm_J(t)).$
Then we may define the cumulative {service potential} between times $s$ and $t$ as
\begin{equation}
    \label{eq: gdef}
    G(s,t)=\int_s^t 1_{\{\wmass(\boldsymbol{Z}(u)) \neq 0\}}\big[\wmass(\boldsymbol{Z}(u))\big]^{-1}\,du.
\end{equation}
For ease of notation, we will often denote the cumulative service potential up to time $t\geq 0$ as $G(t):= G(0,t).$
We will work in the overloaded regime
\(
\rho:=\sum_{j=1}^J \alpha_j \,\mathbb E[\ser^j_1] > 1,
\)
use fluid and diffusion scalings in the Skorokhod \(J_1\) topology on the appropriate path spaces, and allow general renewal inputs with finite third moments.

\medskip

\noindent\textbf{Contributions.}
\begin{enumerate}[label=(C\arabic*)]
\item \emph{Fluid model:} We formulate a deterministic fluid model for \(\boldsymbol{\ssp}(\cdot)\) as a transport equation driven by the limiting service potential function \(\bar{G}(\cdot)\) and the joint service time and patience time laws \(\{\vartheta_j\}\). We prove existence and uniqueness of fluid model solutions for any nontrivial initial condition and show that the fluid–scaled stochastic systems converge to this solution in \(D([0,\infty), \M^J)\), where $\M$ is the set of positive, finite Borel measures on $\R_+^2.$ In the single–class case, our fluid equation reduces to the model of \cite{gromollzwart}.

\item \emph{Diffusion model:} We formulate a diffusion model solution for the limiting diffusion-scaled fluctuations around the fluid model. The diffusion-model solution will be a random rcll path that takes values in the space of tempered distributions. Its projections onto each Schwartz function will be characterized by a system of SDEs that is driven by a heterogeneous Gaussian noise field with a nontrivial covariance structure. We prove uniqueness of diffusion model solutions and convergence of the centered and diffusion-scaled stochastic systems to diffusion-model solutions.
\end{enumerate}
\medskip

\noindent\textbf{Methodology.} First, we develop a renewal–driven martingale decomposition following the framework in \cite{loeser2025diffusionlimitsmeasurevaluedqueueing}. This gives rise to a family of multi–parameter martingales indexed by translated test functions, which we obtain Gaussian limits for using the tools in \cite{loeser2025diffusionlimitsmeasurevaluedqueueing}. 
Together with the theory for stochastic–integral limits for renewal processes developed in \cite{loeser2025diffusionlimitsmeasurevaluedqueueing}, this yields the diffusion limit results.

\medskip

\noindent\textbf{Relation to prior work.}
Gromoll, Kruk, Puha, and Williams analyzed PS queues using compactness-uniqueness arguments, workload lifts, and state–space collapse in settings without reneging or with soft deadlines \cite{gromoll,gromollkruk,gromollpuhawilliams}. Gromoll-Robert-Zwart established a single–class fluid model with reneging \cite{gromollzwart}. Our results {extend} the fluid picture to multiclass PPS and, crucially, provide the {first diffusion approximation} for PS with {firm deadlines} and general renewal data. We note that, in some areas, we have much stronger assumptions than the previous body of work on generally distributed PS queues. In particular, for the diffusion approximation, we have put significant regularity conditions on both the fluid-model initial condition and the incoming patience-time distributions in the form of twice differentiable CDFs and PDFs. For the fluid model, we have restricted to the case of a nonzero initial measure and assumed Lipschitz continuous marginal CDFs for the initial measures and the patience-time distributions in order to more easily obtain uniqueness of solutions. We note that the case of the zero initial measure required significant analysis in the uniqueness of solutions proof in \cite{gromollzwart}.
The author believes that the results in this paper may be able to be extended to less regular inputs and the zero initial condition case, but chose to first explore the fluid and diffusion limits in the current setting.

An important related line of work on diffusion limits for
measure-valued queueing processes is due to Kaspi and Ramanan \cite{KaspiRamanan2013SPDE}. They study a many-server queue with generally
distributed service times, using a measure-valued state descriptor
that tracks the ages of customers in service, and establish a
functional central limit theorem whose infinite-dimensional component
can be characterized by an SPDE. Their analysis, like ours, makes use
of test-function methods and martingales obtained by subtracting
appropriate compensators. Thus, although the queueing model and
scaling regime are different, their work provides an important
precedent for the use of measure-valued martingale methods to obtain
diffusion approximations for queueing systems with generally
distributed primitives.

\medskip

\noindent\textbf{Organization.}
The model and notation are specified in the remainder of this section. We state our main results in \S \ref{sec: main results}.
We prove the fluid model results in \S \ref{sec: fluid proofs}.
We prove the diffusion model results in \S \ref{sec: diffusion proofs}.

\subsection{Notation} We shall use the following notation throughout the paper.
Let $\N$ denote the set of strictly positive integers, $\{1,2,....\},$ and let $\N_0= \N \cup \{0\}$. 
For a positive integer $N,$ let $[N]$ denote the set $\{1,..., N\}.$ 
For $x\in \R$ we denote the positive part of $x$ by $x^+:=x\vee 0.$ 
For a finite set $A \subset\R_+,$ we denote the $i$th smallest element of $A$ by $A_{\{i\}}$. 
Let $\pi_x(x,y):=x$ and $\pi_y(x,y):=y$ denote the component projections on $\R_+^2.$
We denote the zero vector in any vector space by $\mathbf{0}$. 
For a vector $\boldsymbol{x}\in \R^d,$ we write $\boldsymbol{x}>\boldsymbol{0}$ if and only if $x_i>0$ for $i=1,...,d.$  
For a set $B \in \R_+,$ we write $mB$ to denote $\{mx: x \in B\}$.
For $X=\R$ or $X= \R_+$, we denote the set of bounded continuous functions defined on $X$ and taking values in $\R$ by $\mathbf{C}_b(X).$
The set of functions in $\mathbf{C}_b(X)$ that have bounded continuous derivatives up to order $n\geq 1$ is denoted by $\mathbf{C}_b^n(X).$
For a function $f: \R^n \rightarrow \R$, we write $f_i$ to denote the first partial derivative in the $i$th for each $i \in [n].$
For a function $f: \R_+^n\rightarrow \R,$ write $f^+:= 1_{\{\boldsymbol{x}>0\}}f(\boldsymbol{x})$.
The set of Schwartz functions on $X$ will be denoted $\mathscr{S}(X).$
We denote the dual of the Schwartz space on $X$ as $\mathscr{S}'(X).$
For $T \geq 0$ and a bounded function $f:\R_+\rightarrow \R,$ we write $||f||_T$ for $\sup_{t\in[0,T]}|f(t)|$ and $||f||$ for $||f||_{[0,\infty)}.$ 
We take $\sup\emptyset$ to be $0$ and $\inf \emptyset$ to be $+\infty.$
Let $\R_+=[0,\infty),$ and consider it with the Borel $\sigma$-algebra $\mathscr{B}(\R_+).$ 
We denote the set of positive, finite measures on $(\R_+^n, \mathscr{B}(\R_+^n))$ by $\M(\R^n_+).$ 
We endow $\M$ with the topology of weak convergence of measures.
If $\xi \in \M(\R^n_+)$ and $f$ is a Borel measurable function on $\R_+^n$ that is integrable with respect to $\xi$, we let $\langle f, \xi\rangle := \int_{\R_+^n}fd\xi.$ 
For $\mu\in\boldsymbol{M}$, define the y-tailed CDF to be
        $F^\mu(x,t)
        :=
        \mu\bigl([0,x]\times(t,\infty)\bigr),
        \qquad x,t\geq0.$
If $F$ is a function of bounded variation and $g$ is integrable with respect to $\mu_F,$ the Lebesgue-Stieltjes measure associated to the function $F,$ then we denote $\int_{(s,t]} g d\mu_F$ as $\int_s^t gdF.$
We denote the space of functions from $[0,\infty)$ to $\R^d$ that are right continuous with finite left limits by $D([0,\infty),\R^d)$. 
We endow $D([0,\infty),\R^d)$ with the Skorokhod-$J_1$ topology, under which it is a Polish space. 
We denote $\delta_{(x,y)}^+:= 1_{\{x,y>0\}}\delta_{(x,y)}.$
We will commonly denote a vector by using a bold symbol. 
For example, if we have introduced $x_1,...,x_d,$ then $\boldsymbol{x}$ will be $(x_1,...,x_d)^{\bot}.$ Similarly, if we have also introduced $y_1,...,y_d,$ then $\boldsymbol{xy}$ will be $(x_1y_1,...,x_dy_d)^{\bot},$ and so on.
If $\boldsymbol{\nu}\in \M^d$ for some $d \in \N,$ and $\boldsymbol{f}\in \mathscr{B}(\R_+)^d,$ then we denote the vector $(\langle f_1, \nu_1\rangle, ...,\langle f_d, \nu_d\rangle )^{\bot}$ as $\langle \boldsymbol{f}, \boldsymbol{\nu}\rangle.$
Similarly, we denote $(\langle 1, \nu_1\rangle, ...,\langle 1, \nu_d\rangle )^{\bot}$ as $\langle \boldsymbol{1}, \boldsymbol{\nu}\rangle$ and $(\langle \chi, \nu_1\rangle, ...,\langle \chi, \nu_d\rangle )^{\bot}$ as $\langle \boldsymbol{\chi}, \boldsymbol{\nu}\rangle.$

\subsection{Model Setup}
\label{sec: model setup}

\paragraph{(i) Arrivals:} For each $j \in [J]$, let $\inta_0^j$ be the first arrival time to class $j$, and $\{\inta_i^j\}_{i=1}^\infty$ be the i.i.d.\ interarrival times.
Then $\iinta_i^j := \sum_{l=0}^{i-1} \inta_l^j$ is the arrival time of the $i$th job. The cumulative arrival process is
\[
\ap_j(t) := \sup\left\{ i \in \N : \iinta_i^j \leq t \right\}.
\]
The vector of arrival processes is denoted by
\[
\boldsymbol{\ap}(\cdot) = (\ap_1(\cdot), \ldots, \ap_J(\cdot)).
\]

\paragraph{(ii) Service and patience times:} For each class $j \in \J$, we let $\{(\ser_i^j, \pat_i^j)\}_{i=1}^\infty$, with $(\ser_i^j, \pat_i^j) \sim \vartheta_j$, be the i.i.d.\ sequence of initial service and patience times of jobs of class $j$.

\paragraph{(iii) Proportional Processor Sharing Service Discipline:} Each class of jobs $j \in [J]$ is assigned a priority weight $p_j$ such that $\sum_{j=1}^Jp_j=1.$
Then, if $\boldsymbol{\flm} = (\flm_1, \ldots, \flm_J)$ is such that $\flm_j$ is the number of jobs of type $j$ present, then each job of type $j$ will receive service at rate $
\frac{p_j }{\sum_{i=1}^J p_i \flm_i}
.$
\paragraph{(iv) Initial Conditions:} Let $\boldsymbol{\tm}_0 = (\tm_{0,1}, \ldots, \tm_{0,J})$ denote the initial queue lengths. Let $\{(\tilde{\pat}_{i}^j, \tilde{\ser}_i^j)\}_{i=1}^{\infty}$ be the remaining patience time of the $i$th job of type $j$ in the system at time $0$, where $\{(\tilde{\pat}_{i}^j,\tilde{\ser}_i^j)\}_{i=1}^{\infty}$ are i.i.d..

\paragraph{(v) State Descriptor:}
We now more rigorously define the measure-valued process $\ssp_j(t)$ for $j \in [J]$ tracks the remaining service and patience times of jobs in queue $j$ at time $t$
\begin{align}
\ssp_j(t) &:= \sum_{i=1}^{\tm_{0,j}} \delta_{(\tilde{\ser}_i^j-p_jG(0,t),\tilde{\pat}_{i}^j - t)}^+
+ \sum_{i=1}^{\ap_j(t)} \delta_{(\ser_i^j-p_jG(\iinta_i^j,t),\pat_i^j+\iinta_i^j - t)}^+, \hspace{4mm}t\geq0, \numberthis
\label{statespacedescriptordefequation}
\end{align}
with $\boldsymbol{\ssp}(t):= (\ssp_1(t),...,\ssp_J(t))$, where the potential service provided between times $s$ and $t,$ $G(s,t)$ is as defined in \eqref{eq: gdef}. 
In order to obtain a Markovian state descriptor using this measure-valued representation, we let $a_j(t)$ be the time until the next arrival to class $j$ at time $t \geq 0$ and $\boldsymbol{a}(t):= (a_1(t),...,a_J(t))$, $t \geq 0.$ Then we may define the state descriptor to be $(\boldsymbol{\ssp}(t),\boldsymbol{a}(t))$ for $t\geq0.$

We note that the equation \eqref{statespacedescriptordefequation} has some equivalent representations.
First, we may think of it as a mass-transport equation.
We will denote translation by $(a,b)\in \R_+^2$ of a function $f: \R_+^2\rightarrow \R$ as follows:
\begin{equation}
t_{a,b}f(x,y):= 
\begin{cases}
    f((x-a), (y-b)) & x>a, y >b\\
    0 & \text{otherwise}
\end{cases}
    \label{translationnotation}
\end{equation}
Then, it follows from \eqref{statespacedescriptordefequation} that for a Borel-measurable function $f$ and $0 \leq s \leq t,$
\begin{equation}
\langle f, \ssp_j(t)\rangle = \langle t_{p_jG(s,t),t-s}f, \ssp_j(s) \rangle + \sum_{i=A_j(s)+1}^{A_j(t)} \langle t_{p_jG(\iinta_i^j,t),t-U_i^j}f, \delta^+_{(\ser_i^j,\pat_i^j)} \rangle.
\label{masstransportequationdef}
\end{equation}
Note that, almost surely, all measurable functions are integrable with respect to $\bar{\ssp}^m_j(t)$ for each $t\geq 0,$ as it is a finite sum of weighted $\delta$-masses.
This lends itself to the following more classical transport equation formulation:
\begin{lem}
    Let $f\in C_b^1(\R_+^3)$ such that, for each $t\geq 0,$ $f(t,\cdot,\cdot)$ and its first derivatives $\equiv 0$ on $\partial \R_+^2.$ Then, almost surely, for each $t \geq 0,j \in \J$
    \begin{align*}
        \langle f(t,\cdot,\cdot), \ssp_j(t)\rangle &= \langle f(0,\cdot,\cdot), \ssp_j(0) \rangle + \int_0^t \langle f_1(s, \cdot, \cdot), \ssp_j(s) \rangle ds - \int_0^t \langle f_3(s,\cdot,\cdot),\ssp_j(s)\rangle ds\\&- \int_0^t 1_{\{\boldsymbol{\tm}(s) \neq \boldsymbol{0}\}} \frac{p_j}{\wmass(\boldsymbol{\tm}(s))}\langle f_2(s,\cdot,\cdot), \ssp_j(s) \rangle ds
        +\sum_{i=A_j(0)+1}^{A_j(t)} f(\iinta_i^j,\ser_i^j,\pat_i^j)\numberthis\label{differentialmasstransport}
    \end{align*}
\end{lem}
\begin{proof}
 
Fix $T\geq0$ and work on the event
$\Omega_T
:=
\left\{
\sum_{k=1}^J
\bigl(\tm_k(0)+\ap_k(T)\bigr)<\infty
\right\},$
which has probability one. Fix a realization in $\Omega_T$, an
arbitrary $t\in[0,T]$, and $j\in\J$.
For such a realization, there are only finitely many
arrival and departure times in $[0,t]$. Let
\[
0=\tau_0<\tau_1<\cdots<\tau_{N}=t
\]
be the ordered set consisting of $0$, $t$, and all event times in
$(0,t]$.
On each interval $[\tau_i,\tau_{i+1})$, the queue-length vector
$\boldsymbol{\tm}(\cdot)$ is constant. Consequently, there exists a
constant $q_i\geq0$ such that
\[
1_{\{\wmass(\boldsymbol{\tm}(s))\neq0\}}
\frac{1}{\wmass(\boldsymbol{\tm}(s))}
=q_i,
\qquad s\in[\tau_i,\tau_{i+1}),
\]
and hence
\[
G(a,b)=q_i(b-a),
\qquad \tau_i\leq a\leq b\leq\tau_{i+1}.
\]
To obtain equation \eqref{differentialmasstransport}, it suffices to show that, for each $0 \leq i \leq N-1,$ 
\begin{align*}
&\left\langle
f(\tau_{i+1},\cdot,\cdot),\ssp_j(\tau_{i+1})
\right\rangle
-
\left\langle
f(\tau_i,\cdot,\cdot),\ssp_j(\tau_i)
\right\rangle\\
&=
\int_{\tau_i}^{\tau_{i+1}}
\left\langle
f_1(s,\cdot,\cdot),\ssp_j(s)
\right\rangle ds
-p_jq_i
\int_{\tau_i}^{\tau_{i+1}}
\left\langle
f_2(s,\cdot,\cdot),\ssp_j(s)
\right\rangle ds \\&
-
\int_{\tau_i}^{\tau_{i+1}}
\left\langle
f_3(s,\cdot,\cdot),\ssp_j(s)
\right\rangle ds + \sum_{l=A_j(0)+1}^{A_j(t)} 1_{\{\tau_{i+1}=\iinta_l^j\}} f(\iinta_l^j,\ser_l^j,\pat_l^j)
\numberthis \label{eq: intervaltransport}
\end{align*}

For fixed $n \in \N, 0\leq i\leq N-1$, define $\Delta_{i,n}:=\frac{\tau_{i+1}-\tau_i}{n},$ and $
t_{i,m}^n:=\tau_i+m\Delta_{i,n}$ for $m=0,\ldots,n.$ Then we see that
\begin{align}
&\langle f(\tau_{i+1},\cdot,\cdot), \ssp_j(\tau_{i+1})  \rangle - \langle f(\tau_i,\cdot,\cdot), \ssp_j(\tau_i)  \rangle =\\& \hspace{8mm} \lim_{n\rightarrow \infty}\sum_{m=0}^{n-1} \left(\left\langle f\left(t_{i,m}^n + \Delta_{i,n},\cdot,\cdot\right), \ssp_j\left(t_{i,m}^n + \Delta_{i,n}\right)  \right\rangle - \langle f(t_{i,m}^n,\cdot,\cdot), \ssp_j(t_{i,m}^n)  \rangle \right)\label{1}\\
\hspace{8mm} &=\lim_{n\rightarrow \infty}\sum_{m=0}^{n-1} \left(\left\langle f\left(t_{i,m}^n + \Delta_{i,n},\cdot,\cdot\right), \ssp_j\left(t_{i,m}^n + \Delta_{i,n}\right)  \right\rangle - \left\langle f(t_{i,m}^n,\cdot,\cdot), \ssp_j\left(t_{i,m}^n+\Delta_{i,n}\right)  \right\rangle \right)\label{2}\\
&+\lim_{n\rightarrow \infty}\sum_{m=0}^{n-1} \left(\left\langle f\left(t_{i,m}^n ,\cdot,\cdot\right), \ssp_j\left(t_{i,m}^n + \Delta_{i,n}\right)  \right\rangle - \langle f(t_{i,m}^n,\cdot,\cdot), \ssp_j(t_{i,m}^n)  \rangle \right)\label{3}\\
\end{align}
provided that all of the limits exist.
It follows from a standard real analysis argument using the Riemann integrability of $s\rightarrow \langle f_1(s,\cdot,\cdot), \ssp_j(s) \rangle,$ the fact that $\ssp_j(\cdot)$ has bounded total mass on $[0,T]$, and the mean value theorem, that \eqref{2} converges to $\int_{\tau_{i}}^{\tau_{i+1}}\langle f_1(s,\cdot,\cdot), \ssp_j(s)\rangle ds$ (for a detailed version of this real analysis argument, see, e.g. the proof of Lemma 7.2 in \cite{loeserwilliams}).
Applying \eqref{masstransportequationdef}, we see that \eqref{3} is 
\begin{align}
    &\lim_{n\rightarrow \infty}\sum_{m=0}^{n-1} \left(\left\langle f\left(t_{i,m}^n ,\cdot,\cdot\right) ,\ssp_j\left(t_{i,m}^n + \Delta_{i,n}\right)  \right\rangle - \langle f(t_{i,m}^n,\cdot,\cdot), \ssp_j(t_{i,m}^n)  \rangle \right)\\
    &=\lim_{n\rightarrow \infty}\sum_{m=0}^{n-1}\langle t_{p_jq_i\Delta_{i,n},\Delta_{i,n}}f(t_{i,m}^n,\cdot,\cdot)-f(t_{i,m}^n,\cdot,\cdot), \ssp_j(t_{i,m}^n) \rangle\label{4}\\
    &+\lim_{n\rightarrow \infty}\sum_{m=0}^{n-1}\sum_{l=A_j(t_{i,m}^n)+1}^{A_j(t_{i,m}^n +\Delta_{i,n})} \langle t_{p_jG(\iinta_l^{j,m}/m,t_{i,m}^n+\Delta_{i,n}),t_{i,m}^n+\Delta_{i,n}-\iinta_l^{j,m}/m}f(t_{i,m}^n,\cdot,\cdot), \delta^+_{(\ser_l^j,\pat_l^j)}\rangle 
    \end{align}
    again provided that all limits exist.
    Again using the mean value theorem, the boundary conditions on $f$, and
the Riemann integrability of the resulting integrands, we obtain
\begin{align*}
&\lim_{n\to\infty}
\sum_{m=0}^{n-1}
\left\langle
t_{p_jq_i\Delta_{i,n},\,\Delta_{i,n}}
f(t_{i,m}^n,\cdot,\cdot)
-
f(t_{i,m}^n,\cdot,\cdot),
\ssp_j(t_{i,m}^n)
\right\rangle\\
&=
-p_jq_i
\int_{\tau_i}^{\tau_{i+1}}
\left\langle
f_2(s,\cdot,\cdot),\ssp_j(s)
\right\rangle ds
-
\int_{\tau_i}^{\tau_{i+1}}
\left\langle
f_3(s,\cdot,\cdot),\ssp_j(s)
\right\rangle ds.
\end{align*}
   Since there are no arrivals in $(\tau_i,\tau_{i+1})$, only arrivals
at $\tau_{i+1}$ contribute to the second term. For such an arrival,
the spatial translation is zero, while
$t_{i,n-1}^n\uparrow\tau_{i+1}$. Therefore, by continuity of $f$,
the second term converges to
\[
\sum_{\ell=A_j(0)+1}^{A_j(t)}
1_{\{\iinta_\ell^j=\tau_{i+1}\}}
f(\iinta_\ell^j,\ser_\ell^j,\pat_\ell^j).
\]
\end{proof}
\section{Main Results}
\label{sec: main results}
The results in this paper will be in characterizing fluid and diffusion approximations of this system. 
We will be working with a sequence of models as described in \S \ref{sec: model setup} indexed by $m \in \N$.
Throughout the paper, we append the superscript $m$ to denote quantities associated with the $m$th model. For example, the interarrival times for class $j$ are denoted $\{\inta_i^{j,m}\}_{i=1}^{\infty}$.
Some parameters will change with $m$ and others will remain the same.
In particular, the sequences $\{\inta_i^{j,m}\}_{i=1}^{\infty}$ will change with $m,$ but $p_j, j \in \J$ will be fixed, and so will $\pd_j, j \in \J.$
For each $j \in \J,$ there will be an i.i.d. sequence $\{(\ser_i^j,\pat_i^j)\}_{i=1}^{\infty}$ such that the sequence of initial service and patience times for the $m$th system, $\{(\ser_i^{j,m},\pat_i^{j,m})\}_{i=1}^{\infty}$ have the property that $(\ser_i^{j,m},\pat_i^{j,m}):=(\ser_i^j,m\pat_i^j)$ for each $i,m>0.$

The following conditions will be assumed throughout:
\begin{assumption}
\label{assumptions}
\begin{enumerate}[label=(\roman*)]
\item \label{basicassumptions}
For all $j \in [J]$, $m \in \N,$ the arrival rate $\ar_j^m := 1/\mathbb{E}[\inta_1^{j,m}]$, reneging rate $\gamma_j := 1/\mathbb{E}[\pat_1^j]$, and service rate $\sr_j := 1/\mathbb{E}[\ser_1^{j}]$ are all positive and finite.
The expected initial number of jobs in the queue for class $j$, $\mathbb{E}[\tm_{0,j}^m]$, is finite.
Additionally, for each $t \geq 0$, $j \in [J]$,
\[
\sup_{m \in \N} \mathbb{E}[\bar{\ap}_j^m(t)] < \infty. \]
Furthermore, initial service and patience times are positive almost surely. In other words, we assume 
$$\pd_j(\{0\} \times [0,\infty))=\pd_j([0,\infty)\times\{0\}) =0.$$

\item For each $m\in\N$, the collections
\[
\left\{\{\inta_i^{j,m}\}_{i=1}^{\infty}:j\in[J]\right\},
\qquad
\left\{\{(\ser_i^{j,m},\pat_i^{j,m})\}_{i=1}^{\infty}:j\in[J]\right\}
\]
are mutually independent across classes and between the two collections. They are also independent of the initial condition data
$(\boldsymbol{\tm}_0^m,\boldsymbol{a}^m(0))$ and of the collections
\[
\left\{\{(\tilde{\ser}_i^j,\tilde{\pat}_i^j)\}_{i=1}^{\infty}:j\in[J]\right\}.
\]
The collections $\{(\tilde{\ser}_i^j,\tilde{\pat}_i^j)\}_{i=1}^{\infty}$ are also mutually independent across $j\in[J]$. Lastly, we assume that the marginal CDFs of $\pd_j$ are Lipschitz continuous.
\item \label{parametersassumption}
There exists a limiting arrival rate vector $\boldsymbol{\ar} > 0$ such that $\boldsymbol{\ar}^m \to \boldsymbol{\ar}$ as $m \to \infty$. The system is overloaded in the limit, meaning the load parameter
\[
\load := \sum_{j=1}^J \frac{\ar_j}{\sr_j} > 1.
\]

\item \label{fllnassumption1}
For all $j \in [J]$, we assume that
$\frac{\mathbb{E}[\inta_0^{j,m}]}{\sqrt{m}} \to 0$ as $m \to \infty.$

\item \label{fllnassumption2}
For each $j \in [J]$, $\mathbb{E}[\inta_1^{j,m} \cdot 1_{\{\inta_1^{j,m} > m\}}],$ $
\mathbb{E}[\ser_1^{j} \cdot 1_{\{\ser_1^{j} > m\}}],$ $
\mathbb{E}[\pat_1^{j} \cdot 1_{\{\pat_1^{j} > m\}}]$ converge to $0$ as $m \to \infty,$
and the following third moment bounds hold:
\[
 \mathbb{E}[|\ser_1^{j}|^3] < \infty, \qquad
\sup_{m \in \N} \mathbb{E}[|\inta_1^{j,m}|^3] < \infty, \qquad  \mathbb{E}[|\pat_1^{j}|^3] < \infty
\]

\item \label{initialconditionsassumption}
There exists $\epsilon >0$ and a nonzero deterministic vector of continuous measures with Lipschitz continuous marginal CDFs, $\bar{\boldsymbol{\ssp}}_0 = (\bar{\ssp}_{0,1}, \dots, \bar{\ssp}_{0,J})\in \M(\R_+^2)^J,$ such that $\langle x^{1+\epsilon}, \bar{{\ssp}}_{0,j} \rangle,\langle {\pi}_y, \bar{{\ssp}}_{0,j} \rangle < \infty$ for all $j \in [J]$, and
\[
(\bar{\boldsymbol{\ssp}}^m(0), \langle \boldsymbol{x^{1+\epsilon}}, \bar{\boldsymbol{\ssp}}^m(0) \rangle,\langle \boldsymbol{\pi}_y, \bar{\boldsymbol{\ssp}}^m(0) \rangle)
\Rightarrow
(\bar{\boldsymbol{\ssp}}_0, \langle\boldsymbol{{x}^{1+\epsilon}}, \bar{\boldsymbol{\ssp}}_0 \rangle,\langle \boldsymbol{\pi}_y, \bar{\boldsymbol{\ssp}}_0 \rangle), \qquad \text{as } m \to \infty.
\]
In addition, there exists a random variable $\hat{\boldsymbol{\ssp}}_0 \in \mathscr{S}'(\R_+^2)^J$ such that for any $f_1,\dots,f_J \in \mathscr{S}(\R_+^2)$,
\[
(\langle f_1, \hat{\ssp}_1^m(0) \rangle, \dots, \langle f_J, \hat{\ssp}_J^m(0) \rangle)
\Rightarrow
(\langle f_1, \hat{\ssp}_{0,1} \rangle, \dots, \langle f_J, \hat{\ssp}_{0,J} \rangle),
\]
and for any $f \in \mathscr{S}(\R_+^2) \cup \{1_{(0,\infty)\times(0,\infty)}\}$, the functions
\[
F_f^{j,m}(x,y) := \left\langle f((\cdot - x)^+,(\cdot-y)^+), \hat{\ssp}_j^m(0) \right\rangle
\]
converge in probability to a random function
$F_f^j:\R_+^2\to\R$ that is almost surely continuous, where the convergence is in the topology of uniform convergence on compact sets.
\item \label{cltforrenewalassumption}
Let $\sigma_{\ap,j} > 0$ be the limiting standard deviation of $\inta_1^{j,m}$. Define
\[
\hat{\ap}_j^m(t):=\sqrt{m}\left(\bar{\ap}_j^m(t)-\ar_j t\right), \qquad t\geq 0.
\]
Then, for each $j\in[J]$, the diffusion-scaled renewal process
converges in distribution to a Brownian motion with quadratic
variation
\[
\ar_j^3\sigma_{\ap,j}^2t;
\]
see \cite[Theorem 5.11]{chenandyao}.
\end{enumerate}
\end{assumption}
\subsection{Fluid Model Results}
For a sequence of models as described in \S \ref{sec: model setup} indexed by $m \in \N$, the state descriptor for class $j,$ fluid-scaled by $m,$ is defined as follows. For any Borel sets $B,C \subseteq \R_+$, 
\begin{equation}
\bar{\ssp}^{m}_j(t)(B\times C) := \frac{1}{m} \ssp^{m}_j(mt)(B\times mC), \qquad t \geq 0,
\label{scalingwithborelsets}
\end{equation}
or, equivalently, for any bounded Borel measurable function $f: \R_+^2 \to \R$,
\begin{equation}
\langle f, \bar{\ssp}^{m}_j(t) \rangle = \frac{1}{m} \left\langle f\left(\cdot,\frac{1}{m} \cdot \right), \ssp^{m}_j(mt) \right\rangle.
\label{scalingwithfunctions}
\end{equation}

 We use an overbar to denote fluid-scaled processes; for instance, the fluid-scaled arrival process for class $j$ is $\bar{\ap}^m_j(t) := \frac{1}{m} \ap^m_j(mt)$.
 We now define the parameters for our fluid models.
 \begin{defi}[Fluid Model Parameters]
An admissible set of fluid model parameters is a triple
$(\boldsymbol{\ar},\boldsymbol{p},\boldsymbol{\pd})
\in
\R_{>0}^J\times\R_{>0}^J\times\M^J$
such that
$
\sum_{j=1}^J p_j=1,
$
and, for each $j\in\J$, $\pd_j$ is a probability measure satisfying
\[
\langle \pi_x,\pd_j\rangle<\infty,
\qquad
\langle \pi_y,\pd_j\rangle<\infty,
\qquad
\pd_j(\partial\R_+^2)=0.
\]
We further require that
$
\sum_{j=1}^J
\ar_j\langle\pi_x,\pd_j\rangle>1.$
\end{defi}
\begin{defi}[Fluid Model Solution]
\label{notoverdef}
Let $\boldsymbol{\fl}: [0,\infty) \rightarrow \M^J$ be a continuous function. 
Then we say that $\boldsymbol{\fl}$ is a fluid model solution for fluid model parameters $(\boldsymbol{\ar},\boldsymbol{p},\boldsymbol{\pd})$ 
and initial condition $\boldsymbol{\fl}_0 = (\fl_{0,1},...,\fl_{0,J}),$ a nonzero vector of atomless measures, if 

\begin{enumerate}
\item $\boldsymbol{\fl}(0) = \boldsymbol{\fl}_0,$ \label{startsatICdef}
\item $ \fl_j(t)(\partial \R_+^2)  =0$ for each $t\geq0,j \in \J,$
\label{doesntchargeorigindef}
\item for each $t >0,$ $\langle 1, \fl_j(t)\rangle >0$ for some $j \in [J],$ 
\label{overloadedconditiondef}
\item \label{fluidequationdef} and for each $f \in  \mathbf{C}^1_b(\R_+^2) $ such that $f$ and its first derivatives are $0$ on $\partial \R_+^2,$ $j \in \J, t\geq 0,$
\begin{align*}
\langle f, \fl_j(t)\rangle &= \langle f, \fl_j(0) \rangle -\int_0^t \left\langle f_2, \fl_j(s) \right\rangle ds  \quad - \int_0^t    \frac{p_j\langle f_1, \fl_j(s)\rangle}{\sum_{i=1}^Jp_i\langle 1, \fl_i(s) \rangle } ds \\&\hspace{8mm}+ \ar_j  \langle f,  \pd_j\rangle t. \numberthis \label{fluidlimiteqn}
\end{align*}

\end{enumerate}
\end{defi}
\begin{rem}
    One can check that this fluid equation has an equivalent representation 
    as a ``transport equation," which would be a limiting version of \eqref{masstransportequationdef}. 
    In this formulation, in the single-class case, we see that our fluid limit equation is equivalent to that given in \cite{gromollzwart}.
In particular, for each $f \in  \mathbf{C}_b(\R_+^2) $ $j \in \J, t\geq s\geq 0,$
    \begin{equation}
        \langle f, \fl_j(t)\rangle = \langle t_{p_j\bar{G}(s,t),t-s}f, \fl_j(s)\rangle + \int_s^t\ar_j \langle t_{p_j\bar{G}(r,t),t-r}f, \pd_j \rangle dr 
        \label{fluidtransporteqn}
    \end{equation}
    where $\bar{G}(s,t):= \int_s^t \frac{1}{\sum_{i=1}^Jp_i\langle 1, \fl_i(r) \rangle }dr$ for $0 \leq s\leq t.$

    The proof of this follows the form of e.g., the proof of Lemmas 4.1 and 4.3 in \cite{gromollpuhawilliams} or the proof of Lemma 4.1 of \cite{loeserwilliams}, but is only sketched here because the techniques are the same as in those two papers.
    The method would be to extend equation \eqref{fluidlimiteqn} to an integral equation that holds for functions $f\in \mathbf{C}_b^1(\R_+^3)$, such that, for each $t\geq 0,$ the function $f(t,\cdot,\cdot)$ and its first derivatives in $x$ and $y$ are $0$ on $\partial \R_+^2.$
    In particular, we'll find that for such an $f,$
    \begin{align*}
    \langle f(t,\cdot,\cdot), \fl_j(t)\rangle &= \langle f(0,\cdot,\cdot), \fl_j(0) \rangle + \int_0^t \langle f_1(s,\cdot,\cdot), \fl_j(s)\rangle ds -\int_0^t \left\langle f_3(s,\cdot,\cdot), \fl_j(s) \right\rangle ds  \\& - \int_0^t    \frac{p_j\langle f_2(s,\cdot,\cdot), \fl_j(s)\rangle}{\sum_{i=1}^Jp_i\langle 1, \fl_i(s) \rangle } ds 
    + \int_0^t \ar_j  \langle f(s,\cdot,\cdot),  \pd_j\rangle ds. 
    \numberthis
\label{fluidtransportequationdifferential}
\end{align*}
We then substitute in equations of the form $f_t(r,x,y)= t_{p_j\bar{G}(r,t),t-r}g(x,y)$ where $g(x,y)$ is a bounded continuous function such that it and its first partials are $0$ on $\partial \R_+^2.$ 
In this manner, one obtains \eqref{fluidtransporteqn} for such a function $g.$
The author notes that, by choosing an appropriate approximating sequence $g_n \rightarrow 1_{(x,\infty)\times(y,\infty)}$ and noting that $\fl_j$ does not charge $\partial \R_+^2$, one can substitute $1_{(x,\infty)\times(y,\infty)}$ (see the aforementioned argument in \cite{gromollpuhawilliams}).
Applying the monotone class theorem and bounded convergence, one obtains \eqref{fluidtransporteqn} for any bounded Borel measurable function $f.$
\end{rem}
\begin{thm}
    Let $(\boldsymbol{\ar},\boldsymbol{p},\boldsymbol{\pd}) $ be a set of fluid model parameters 
and let $\boldsymbol{\fl}_0 \in \M^J$. Assume further that $\boldsymbol{\pd},\boldsymbol{\fl}_0$ are nonzero vectors of measures whose marginal CDFs are Lipschitz continuous, and that there exists $\epsilon >0$ such that $\langle \boldsymbol{x}^{1+\epsilon},\boldsymbol{\fl}_0\rangle, \langle \pi_y,\boldsymbol{\fl}_0 \rangle <\boldsymbol{\infty}.$ Then there exists a unique fluid model solution for these parameters and initial condition.
\end{thm}

\begin{thm}
    Let $\{\bar{\boldsymbol{\ssp}}^m(\cdot)\}_{m=1}^{\infty}$ be a sequence of fluid-scaled models that satisfies Assumption \ref{assumptions}. Then $\{\bar{\boldsymbol{\ssp}}^m(\cdot)\}_{m=1}^{\infty}$ converges in distribution to $\boldsymbol{\fl}(\cdot),$ where
$\boldsymbol{\fl}(\cdot)$ is the unique fluid model solution with
initial condition $\bar{\boldsymbol{\ssp}}_0$.
    \label{thm: fluid limit}
\end{thm}
\subsection{Diffusion Model Results}
Using the fluid model results, we may apply a diffusion scaling as follows.
For a sequence of models as described in \S \ref{sec: model setup} indexed by $m \in \N$, we define the diffusion-scaled state descriptor
\begin{equation}
\hat{\boldsymbol{\ssp}}^m(\cdot) = \sqrt{m}(\bar{\boldsymbol{\ssp}}^m(\cdot)-\boldsymbol{\fl}(\cdot)), 
\label{diffusiondefinitioneqn}
\end{equation}
where, $\boldsymbol{\fl}$ is the unique fluid model solution with initial condition $\bar{\boldsymbol{\ssp}}_0(\omega).$

Following Theorem 5.3 2) of \cite{mitoma} and the extension of that
theorem to the interval $[0,\infty)$ in Remark (R.2.2) of the same
work, convergence of the sequence
\[
\{\hat{\boldsymbol{\ssp}}^m(\cdot)\}_{m=1}^{\infty}
\]
in
\[
D([0,\infty),\mathscr S'(\R_+^2)^J)
\]
may be obtained by combining the following two steps:
\begin{enumerate}
    \item showing tightness of the $\R^J$-valued projections
    \[
    \left\{
    \left\langle \boldsymbol f,
    \hat{\boldsymbol{\ssp}}^m(\cdot)\right\rangle
    \right\}_{m=1}^{\infty}
    \]
    for each
    $\boldsymbol f\in\mathscr S(\R_+^2)^J$; and

    \item identifying, for every finite collection
    $f_1,\ldots,f_n\in\mathscr S(\R_+^2)$ and
    $t_1,\ldots,t_n\in[0,\infty)$, the joint limiting law of the
    corresponding finite collection of classwise projections
    \[
    \left(
    \left\langle f_a,\hat{\ssp}_j^m(t_a)\right\rangle:
    a\in[n],\ j\in\J
    \right).
    \]
\end{enumerate}
Theorem \ref{tightnessresult} establishes the first condition.
Theorem \ref{lhatconvergencethm} establishes uniqueness of the
limiting finite-dimensional distributions in the second condition.
Consequently, every subsequential limit of
$\{\hat{\boldsymbol{\ssp}}^m(\cdot)\}_{m=1}^{\infty}$ has the same
law. Together with tightness, this implies
\[
\hat{\boldsymbol{\ssp}}^m(\cdot)
\Rightarrow
\hat{\boldsymbol{\ssp}}(\cdot)
\qquad\text{in}\qquad
D([0,\infty),\mathscr S'(\R_+^2)^J),
\]
where the finite-dimensional distributions of
$\hat{\boldsymbol{\ssp}}(\cdot)$ are characterized by
Theorem \ref{lhatconvergencethm}.

Similar to the bar denoting fluid-scaling, we use a hat to denote diffusion-scaling for relevant processes.
Before defining our diffusion model solutions, we define our driving diffusive noise field.
\begin{thm}
    \label{Ydefthm}
    Let $\boldsymbol{\fl}$ be a fluid model solution with a nonzero initial condition $\boldsymbol{\fl}_0$ and $T >0.$
    For each finite collection of functions \(f_1,\ldots,f_n\in\mathscr{S}(\R_+^2)\), there exists a unique in distribution, continuous multiparameter $\R^{nJ}$-valued process
    \[
\left\{
\hat{\mart}_{f_a}^{\ap_j,j}(r,t),
a\in[n],\ j\in[J],\ 0\le r\le t\le T
\right\}
\]
    \normalsize
    with finite dimensional distributions that are equal to those of the mean-zero Gaussian process with covariances 
\begin{align*}
    &\operatorname{Cov}
\left(\hat{\mart}_{f_a}^{\ap_j,j}(r_1,t_1),\hat{\mart}_{f_b}^{\ap_j,j}(r_2,t_2)\right)\\&= \int_0^{r_1\wedge r_2} \ar_j \left\langle
\left(t_{p_j\bar{G}(u,t_1),\,t_1-u}f_a\right)
\left(t_{p_j\bar{G}(u,t_2),\,t_2-u}f_b\right),
\pd_j
\right\rangle du\\
    &-\int_0^{r_1\wedge r_2} \ar_j\langle t_{p_j\bar{G}(u,t_1),t_1-u}f_a, \pd_j\rangle \langle t_{p_j\bar{G}(u,t_2),t_2-u}f_b, \pd_j \rangle du,
\end{align*}
and
$$
\operatorname{Cov}
\left(
\hat{\mart}_{f_a}^{\ap_j,j}(r_1,t_1),
\hat{\mart}_{f_b}^{\ap_k,k}(r_2,t_2)
\right)
=0,
\qquad j\neq k.
$$
Furthermore, if $\pd_j$ has a density $g_j \in \mathbf{C}_b^2(\R_+^2)$ for each $j \in \J,$ the same construction applies to finite collections
$f_1,\ldots,f_n\in \mathscr{S}(\R_+^2)\cup\{1\}$.
\end{thm}
\begin{defi}
    \label{Udefdef}
    Let $\boldsymbol{\fl}$ be a fluid model solution with a nonzero initial condition $\boldsymbol{\fl}_0.$
    Assume also that $\pd_j$ has a density $g^j \in \mathbf{C}_b^2(\R_+^2)$ for each $j \in \J,$
    Then, take any collection of test functions
\(f_1,\ldots,f_n\in\mathscr{S}(\R_+^2)\cup\{1\}\) and a jointly specified realization of the
limiting driving noise processes
\[ \{\hat{\ap}_j\}_{j \in \J},\quad
\{F_{f_a}^{j}\}_{a\in[n],j\in\J},
\quad
\{\hat{\mart}_{f_a}^{\ap_j,j}\}_{a\in[n],j\in \J,},
\]
where the martingale field has the joint law specified in
Theorem~\ref{Ydefthm} and the other processes are as defined in Assumption \ref{assumptions} \eqref{initialconditionsassumption}-\eqref{cltforrenewalassumption}.
Furthermore, assume that the three processes above form a mutually independent collection of processes.
   Then we may jointly define the driving noise field for the test functions $f_1,...,f_n$ and $j \in \J$ as follows:
    \begin{align*}
    U_{f_a}^{j}(t) &= F_{f_a}^{j}(p_j\bar{G}(t),t) \\&+\hat{\mart}^{\ap_j,j}_{f_a}(t,t) 
     \\&+\int_0^t\langle t_{p_j\bar{G}(s,t),t-s}f_a, \pd_j \rangle d\hat{\ap}_j(s),
     \numberthis \label{ujmlimit}
    \end{align*}
\end{defi}
\begin{rem}
    We remark that, for each $a \in [n], j \in [J],$ if we further assume that $\fl_j(0)$ has a density $h^j \in \mathbf{C}_b^2(\R_+^2)$ and that both $\pd_j,\fl_j(0)$ have CDFs in $C_b^1(\R_+^2),$ then the process $U_{f_a}^j(\cdot)$ is continuous.
    This is an immediate consequence of the fact that it is equal in distribution to the limit of a C-tight sequence of processes, which will be proved in Lemma \ref{uconvlemma}.
\end{rem}
 \begin{thm}[Tightness of the PPS diffusion-scaled state descriptors]
\label{tightnessresult}
Under Assumption~\ref{assumptions}, suppose that $\pd_j,\fl_{0,j}$ have densities $g^j,h^j$ and y-tailed CDFs $F^{\pd_j},F^{\fl_{0,j}}$ in $\mathbf{C}_b^2(\R_+^2)$ for each $j \in \J$.
Then the sequence
\[
    \{\hat{\boldsymbol{\mathcal Z}}^m(\cdot)\}_{m=1}^{\infty}
\]
of diffusion-scaled PPS state descriptors is tight in
$D([0,\infty),\mathscr S'(\R_+^2)^J)$. Moreover, for every
$\boldsymbol f\in \mathscr{S}(\R_+^2)^J$, the sequence
\[
\left\{
\left(
    \langle f_1,\hat{{\ssp}}_1^m(\cdot)\rangle,
    \ldots,
    \langle f_J,\hat{{\ssp}}_J^m(\cdot)\rangle,
    \wmass(\hat{\boldsymbol{\tm}}^m(\cdot))
\right)
\right\}_{m=1}^{\infty}
\]
is $C$-tight in $D(\mathbb R_+,\mathbb R^{J+1})$.
\end{thm}
Next, we introduce a system of stochastic differential equations that will be related to the diffusion limit of $\{\hat{\boldsymbol{\mathcal Z}}^m(\cdot)\}_{m=1}^{\infty}.$
\begin{defi}[Diffusion Model Solution for Total Mass]
\label{def: diffusion model solution for total mass}
Let $\boldsymbol{U}_{1}(r,t)$ be a realization of the multiparameter process defined in Definition \ref{Udefdef}. For $j \in \J,$ suppose that $\pd_j,\fl_{0,j}$ have densities $g^j,h^j$ and y-tailed CDFs $F^{\pd_j},F^{\fl_{0,j}}$ in $\mathbf{C}_b^2(\R_+^2)$.
Define the system of equations
    \begin{align*}
        \hat{G}(\cdot) = -\int_0^{\cdot}\frac{\wmass(\hat{\boldsymbol{\tm}}(u))}{\wmass(\boldsymbol{\flm}(u))^2}du,
        \numberthis
        \label{LGsystem1}
    \end{align*}
    \begin{align*}
\wmass(\hat{\boldsymbol{\tm}}(t))& = \sum_{j=1}^J p_j\bigg(U_{1}^{j}(t)  -p_j\hat{G}(t)\int_t^{\infty} h^j(p_j\bar{G}(t),y)dy
     \\&-\int_0^t\ar_jp_j\hat{G}(s,t)\int_{t-s}^{\infty} g^j(p_j \bar{G}(s,t),y)dy ds\bigg),
    \numberthis 
    \label{LGsystem2}
    \end{align*}
    where $\hat{G}(s,t):=\hat{G}(t)-\hat{G}(s).$
    Then any solution to the system \eqref{LGsystem1}-\eqref{LGsystem2} is a diffusion model solution for total mass.
\end{defi}
\begin{thm}[Uniqueness of Diffusion Model Solutions for Total Mass]
\label{Llimit}
For $j\in\J$, suppose that $\pd_j$ and $\fl_{0,j}$ have densities
$g^j,h^j$ and y-tailed CDFs $F^{\pd_j},F^{\fl_{0,j}}$ in
$\mathbf{C}_b^2(\R_+^2)$. For any realization of the driving process $\boldsymbol{U}_1$
defined in Definition~\ref{Udefdef} corresponding to the fluid model
solution $\bar{\boldsymbol{\ssp}}$, the system
\eqref{LGsystem1}--\eqref{LGsystem2} has at most one continuous solution
\[
\bigl(\wmass(\hat{\boldsymbol{\tm}}(\cdot)),\hat G(\cdot)\bigr).
\]
Moreover, every subsequential limit of
\[
\left\{
\left(
\wmass(\hat{\boldsymbol{\tm}}^m(\cdot)),
\hat G^m(\cdot)
\right)
\right\}_{m=1}^{\infty}
\]
can be jointly realized with a process having the law of
$\boldsymbol{U}_1$ specified in Definition~\ref{Udefdef}
corresponding to the fluid model solution
$\bar{\boldsymbol{\ssp}}$, so that it is a continuous solution of
\eqref{LGsystem1}--\eqref{LGsystem2}.
Consequently, all subsequential limits have the same distribution.
\end{thm}

\begin{thm}[Uniqueness and characterization of projected PPS diffusion limits]
\label{lhatconvergencethm}
Let $f_1,\ldots,f_n\in \mathscr{S}(\mathbb R_+^2)$ and suppose that,
for each $j\in\J$, $\pd_j$ and $\fl_{0,j}$ have densities $g^j,h^j$
and y-tailed CDFs $F^{\pd_j},F^{\fl_{0,j}}$ in
$\mathbf{C}_b^2(\R_+^2)$.

Let
\[
\left(
\langle f_1,\hat{\boldsymbol{\ssp}}(\cdot)\rangle,
\ldots,
\langle f_n,\hat{\boldsymbol{\ssp}}(\cdot)\rangle,
\wmass(\hat{\boldsymbol{\tm}}(\cdot)),
\hat G(\cdot)
\right)
\]
be any subsequential limit of
\[
\left\{
\left(
\langle f_1,\hat{\boldsymbol{\ssp}}^m(\cdot)\rangle,
\ldots,
\langle f_n,\hat{\boldsymbol{\ssp}}^m(\cdot)\rangle,
\wmass(\hat{\boldsymbol{\tm}}^m(\cdot)),
\hat G^m(\cdot)
\right)
\right\}_{m=1}^{\infty}.
\]
Then this subsequential limit can be jointly realized with a process
\[
(\boldsymbol{U}_{f_1}(\cdot),\ldots,
\boldsymbol{U}_{f_n}(\cdot),\boldsymbol{U}_1(\cdot))
\]
having the law specified in Definition~\ref{Udefdef}
corresponding to the fluid model solution
$\bar{\boldsymbol{\ssp}}$, such that
$(\wmass(\hat{\boldsymbol{\tm}}(\cdot)),\hat G(\cdot))$
is a continuous solution of \eqref{LGsystem1}--\eqref{LGsystem2}
driven by $\boldsymbol U_1$.

Moreover, for $1\leq i\leq n$, $1\leq j\leq J$, and $t\geq0$,
\begin{align*}
\langle f_i,\hat{\ssp}_j(t)\rangle
&=
U_{f_i}^j(t)
+p_j\hat{G}(t)
\int_{\R_+^2}
h^j_1(x+p_j\bar{G}(t),y+t)f_i(x,y)\,d\lambda
\\
&\quad+
\int_0^t\int_{\R_+^2}
\ar_jp_j\hat{G}(s,t)
g^j_1(x+p_j\bar{G}(s,t),y+t-s)
f_i(x,y)\,d\lambda\,ds.
\numberthis
\label{rjlimitdef}
\end{align*}
Here $\lambda$ denotes Lebesgue measure on $\R_+^2$. Consequently, the joint law of every such subsequential limit is
uniquely determined.
\end{thm}
\section{Proofs of Fluid Limit Results}
\label{sec: fluid proofs}
For the entirety of this section, we will be working with a sequence of fluid-scaled models $\{\bar{\boldsymbol{\ssp}}^m(\cdot)\}_{m=1}^{\infty}$ that satisfies Assumption \ref{assumptions}, aside from the third moment bounds in \ref{fllnassumption2}, the second claim in \ref{initialconditionsassumption}, and the entirety of \ref{cltforrenewalassumption}, which are only required for the diffusion limit.
\subsection{Uniqueness Of Fluid model Solutions}
\begin{prop}[Uniqueness of Fluid Model Solutions]
Let $(\boldsymbol{\ar}, \boldsymbol{p}, \boldsymbol{\pd})$ be fluid model parameters and $\boldsymbol{\fl}_0\in \M^J.$ Let $\boldsymbol{\pd},\boldsymbol{\fl}_0$ be nonzero vectors of measures with Lipschitz continuous marginal CDFs.
    Let $\boldsymbol{\fl}(\cdot), \tilde{\boldsymbol{\fl}}(\cdot)$ be two fluid model solutions with these parameters and the same initial condition.
    Then $\boldsymbol{\fl}(\cdot) = \tilde{\boldsymbol{\fl}}(\cdot).$
\end{prop}
\begin{proof}
    It is clear from the form of \eqref{fluidtransporteqn} that the fluid model solutions $\boldsymbol{\fl}(\cdot)$ and $\tilde{\boldsymbol{\fl}}(\cdot)$ are fully determined by the initial condition and their associated $\bar{G}(\cdot)$ and $\tilde{G}(\cdot)$ functions, respectively.
    Thus, it suffices to prove $\bar{G}(\cdot)=\tilde{G}(\cdot)$.

    Define $t_{\epsilon}^*:= \inf\{t:\wmass(\boldsymbol{\flm}(t)) \wedge \wmass(\tilde{\boldsymbol{\flm}}(t)) \leq \epsilon\}.$ Then, the function $f(x):= 1/x$ is Lipschitz on $[\epsilon, \infty)$ with Lipschitz constant $C_{\epsilon},$ and so for each $t < t_{\epsilon}^*$ we have
    \begin{align*}
        |\bar{G}(t)-\tilde{G}(t)| &\leq \int_{0}^t \bigg|\frac{1}{\wmass(\boldsymbol{\flm}(s))}-\frac{1}{\wmass(\tilde{\boldsymbol{\flm}}(s))} \bigg|ds\\
        & \leq \int_0^t C_{\epsilon} |\wmass(\boldsymbol{\flm}(s)) - \wmass(\tilde{\boldsymbol{\flm}}(s))|ds.
    \end{align*}
    Let $C_{1,i}, C_{2,i}$ be the Lipschitz constants for the marginal of $\fl_{0,i}$ and $\pd_{i}$ in the $x$-direction, respectively, for each $i \in \J.$
    Now, applying \eqref{fluidtransporteqn} and, for notational clarity, writing $H_-(s,t):= \min\{\tilde{G}(s,t),\bar{G}(s,t)\}, H_+(s,t):=\max\{\tilde{G}(s,t),\bar{G}(s,t)\}$, we obtain
   
        \begin{align*}
|\bar{G}(t)-\tilde{G}(t)|
&\leq
\int_0^t
C_{\epsilon}
\sum_{i=1}^J
p_i
\fl_{0,i}
\left(
\bigl(p_iH_-(0,s),p_iH_+(0,s)\bigr]
\times(s,\infty)
\right)
\,ds\\
&\quad+
\int_0^t
C_{\epsilon}
\sum_{i=1}^J
p_i
\int_0^s
\ar_i
\pd_i
\left(
\bigl(p_iH_-(r,s),p_iH_+(r,s)\bigr]
\times(s-r,\infty)
\right)
\,dr\,ds\\
&\leq
\int_0^t
C_{\epsilon}
\sum_{i=1}^J
p_i^2C_{1,i}
|\bar{G}(s)-\tilde{G}(s)|
\,ds\\
&\quad+
\int_0^t
C_{\epsilon}
\sum_{i=1}^J
p_i^2\ar_iC_{2,i}
\int_0^s
\left(
|\bar{G}(s)-\tilde{G}(s)|
+
|\bar{G}(r)-\tilde{G}(r)|
\right)
\,dr\,ds\\
&\leq
\int_0^t
C_{\epsilon}
\sum_{i=1}^J
p_i^2C_{1,i}
|\bar{G}(s)-\tilde{G}(s)|
\,ds\\
&\quad+
C_{\epsilon}
\sum_{i=1}^J
p_i^2\ar_iC_{2,i}
\int_0^t
\left(
s|\bar{G}(s)-\tilde{G}(s)|
+
(t-s)|\bar{G}(s)-\tilde{G}(s)|
\right)
\,ds.
\end{align*}
    Then, uniqueness follows from Gr\"onwall's inequality on $[0,t_{\epsilon}^*).$ Since $\epsilon>0$ was chosen arbitrarily and fluid model solutions are continuous and nonzero, this completes the proof.
\end{proof}
Next, we prove a bound on the mass that is near the boundary in the prelimit. This will be useful in proving C-tightness of the measure-valued process.
\begin{lem}
\label{bdeltalemma}
    Let $B_{\delta}:= \{(x,y)\in \R_+^2:x \text{ or }y \leq \delta\}.$
    Then for $j \in \J,T > 0,\epsilon>0$
    \begin{equation}
        \lim_{\delta\rightarrow 0} \limsup_{m\rightarrow \infty } P\{\sup_{t\in [0,T]}\bar{\ssp}_j^m(t)(B_{\delta})\geq \epsilon\} = 0
        \label{bdeltabound}
    \end{equation}
\end{lem}
\begin{proof}
    By \eqref{masstransportequationdef}, we have that for $j \in \J,t,\delta >0,$
\begin{align*}
   &\lim_{\delta \rightarrow 0}\limsup_{m\rightarrow \infty} P(\sup_{t\in[0,T]} \bar{\ssp}_j^m(t)(B_{\delta})\geq \epsilon )\\&\leq  \lim_{\delta \rightarrow 0}\limsup_{m\rightarrow \infty} P(\sup_{x \geq 0 }\langle 1_{(x,x+\delta]\times(0,\infty)}, \bar{\ssp}_j^m(0)\rangle\geq \epsilon/4) 
   \\&+  \lim_{\delta \rightarrow 0}\limsup_{m\rightarrow \infty} P(\sup_{y \geq 0 }\langle 1_{(0,\infty)\times(y,y+\delta]}, \bar{\ssp}_j^m(0)\rangle\geq \epsilon/4) \\&+\lim_{\delta \rightarrow 0}\limsup_{m\rightarrow \infty} P\left(\sup_{t \in [0,T]}\frac{1}{m}\sum_{i=1}^{m\bar{\ap}^m_j(t)}1_{\{\pat_i^j \in (t-\iinta_i^{j,m}/m,t+\delta-\iinta_i^{j,m}/m]\}}\geq \epsilon/4\right)\\&+
    \lim_{\delta \rightarrow 0}\limsup_{m\rightarrow \infty}P\left(\sup_{t \in [0,T]}\frac{1}{m}\sum_{i=1}^{m\bar{\ap}^m_j(t)}1_{\{\ser_i^j \in (p_j\bar{G}^m(\iinta_i^{j,m}/m,t),p_j\bar{G}^m(\iinta_i^{j,m}/m,t)+\delta]}\geq \epsilon/4\right)
 \numberthis \label{controlledoscdifference}
\end{align*}
Therefore, we need only prove each term on the RHS above goes to zero.
We do so now.
Taking a Skorokhod representation, we see that there is a sequence $\nu_j^m\sim \bar{\ssp}^m_{0,j}$ that converges almost surely to  $\nu_j \sim \bar{\ssp}_{0,j}.$
Because almost sure convergence implies convergence in probability, we see that for each $\delta >0,$ ${\liminf_{m\rightarrow 
\infty} P(d(\nu_j^m,\nu_j) <\delta) = 1}.$
Therefore, using the form of the L\'evy-Prokhorov metric on $\M,$ we conclude that
\small
\begin{equation}
\liminf_{m\rightarrow 
\infty} P\{\nu_j^m((x,x+\delta]\times(0,\infty))\leq \nu_j(((x-\delta)^+,x+2\delta]\times(0,\infty) +2\delta \text{ for each }(x,y)\in \R_+^2\} = 1.
\label{ctnp2}
\end{equation}
\begin{equation}
\liminf_{m\rightarrow 
\infty} P\{\nu_j^m((0,\infty)\times(y,y+\delta])\leq \nu_j((0,\infty)\times(y-\delta)^+,y+2\delta]) +2\delta \text{ for each }(x,y)\in \R_+^2\} = 1.
\label{ctnp6}
\end{equation}
\normalsize
Because we have assumed that the joint CDF of $\bar{\ssp}_{0,j},$ and thus $\nu_j,$ is continuous almost surely, for any $a>0,$ \begin{equation}\lim_{\delta \rightarrow 0}P\left\{ \sup_{x\in \R_+} \langle 1_{[x,x+\delta]\times (0,\infty)}, \nu_j\rangle\vee \sup_{y \in \R_+}\langle 1_{(0,\infty)\times [y,y+\delta]}, \nu_j\rangle < a\right\} =1.\label{ctnp3}\end{equation} 
(For a proof of the equivalence of the above statement with the marginals satisfying the no atoms condition, see Lemma A.1. in \cite{gromollpuhawilliams}.)
We conclude that the first two terms on the RHS of \eqref{controlledoscdifference} are both zero.
We now move to the fourth term on the RHS of \eqref{controlledoscdifference}.

In order to prove the claim, we need to show that, for each $\eta>0,$ there exists a $\delta'$ such that for all $\delta< \delta',$ $\limsup_{m\rightarrow \infty}P\left(\sup_{t \in [0,T]}\frac{1}{m}\sum_{i=1}^{m\bar{\ap}^m_j(t)}1_{\{\ser_i^j \in (p_j\bar{G}^m(\iinta_i^{j,m}/m,t),p_j\bar{G}^m(\iinta_i^{j,m}/m,t)+\delta]}\geq \epsilon/4\right)<\eta.$
Fix $\delta>0$ and define
\[
K_\delta:=\left\lceil\frac{T}{\delta}\right\rceil,
\qquad
t_k^\delta:=(k\delta)\wedge T,
\qquad
k=0,\ldots,K_\delta.
\]
Define
\[
I_k^\delta:=[t_k^\delta,t_{k+1}^\delta),
\qquad
k=0,\ldots,K_\delta-2,
\]
and define the final interval by
\[
I_{K_\delta-1}^\delta
:=
[t_{K_\delta-1}^\delta,T].
\]
Then $\{I_k^\delta\}_{k=0}^{K_\delta-1}$ is a partition of
$[0,T]$ into nonempty intervals satisfying
\[
|I_k^\delta|\leq\delta,
\qquad
k=0,\ldots,K_\delta-1.
\]

Then, we see that, pathwise,
\begin{align*}
    &\sup_{t \in [0,T]}\frac{1}{m}\sum_{i=1}^{m\bar{\ap}^m_j(t)} 1_{\{\ser_i^j \in (p_j\bar{G}^m(\iinta_i^{j,m}/m,t),p_j\bar{G}^m(\iinta_i^{j,m}/m,t)+\delta]\}}\\
    &\leq \sum_{k=0}^{  K_{\delta}-1 }  \sup_{a \geq 0}\frac{1}{m}\sum_{i=1}^{m\bar{\ap}^m_j(T)} 1_{\{\iinta_i^{j,m}/m \in I_k^{\delta}\}}1_{\{\ser_i^j \in (a,a+2\delta]\}}
    \numberthis
    \label{eq: final inequality for bdeltabound}
\end{align*}
because jobs are served at a rate less than or equal to $1,$ and thus, for $u,u'\in I_k^{\delta},$  $|p_j\bar{G}^m(u,t)-p_j\bar{G}^m(u',t)| \leq \delta$ for any choice of $\delta \geq 0,$ $k=0,1,2,....$
Now, defining the number of jobs in the $m$th system that arrives on the time interval $I_k^{\delta}$ as follows:
\begin{equation}
    \label{eq: Nkdef}
    N_k^{\delta,m}:=\sum_{i=1}^{m\bar{\ap}^m_j(T)} 1_{\{\iinta_i^{j,m}/m \in I_k^{\delta}\}},
\end{equation}
we claim that
\begin{equation}
    X_k^{\delta,m}(\cdot):=\frac{1}{N_k^{\delta,m}}\sum_{i=1}^{m\bar{\ap}^m_j(T)} 1_{\{\iinta_i^{j,m}/m \in I_k^{\delta}\}}1_{\{\ser_i^j \in [0,\cdot]\}}\rightarrow X_k^{\delta}(\cdot):= \pd_j([0,\cdot]\times (0,\infty)) \label{eq: probabilitylimit}
\end{equation}
in probability.
Indeed, let $S_k^{\delta,m}:= \{i:\iinta_i^{j,m}/m\in I_{k}^{\delta}\}$ and $Y_k^{\delta,m}(\cdot):=\frac{1}{m}\sum_{i=1}^m 1_{\{\ser_i^j \in [0,\cdot]\}}.$
Then, using the independence of the interarrival and service times, we see that
\begin{align*}
    &P(\sup_{a\geq 0}|X_k^{\delta,m}(a)-X_k^{\delta}(a)|>\epsilon)\\& = \sum_{L=1}^{\infty} P(\sup_{a\geq 0}|X_k^{\delta,m}(a)-X_k^{\delta}(a)|>\epsilon|N_k^{\delta,m} = L)P(N_k^{\delta,m} = L)\\
    &=\sum_{L=1}^{\infty} P(\sup_{a\geq 0}|Y_k^{\delta,L}(a)-X_k^{\delta}(a)|>\epsilon)P(N_k^{\delta,m} = L)\\
    &\leq \sup_{L\geq M}P(\sup_{a\geq 0}|Y_k^{\delta,L}(a)-X_k^{\delta}(a)|>\epsilon)+ P(N_k^{\delta,m} < M)
\end{align*}
for any choice of $\epsilon>0, M \in \N.$
Taking the limit as $m \rightarrow \infty$ on both sides applying any standard concentration inequality for the first term (e.g., the Dvoretzky–Kiefer–Wolfowitz inequality), and using the fact that, by the weak law of large numbers, $\frac{N_k^{\delta,m}}{m}\rightarrow |I_k|\ar_j$ in probability, the convergence in \eqref{eq: probabilitylimit} is proven.

Returning to \eqref{eq: final inequality for bdeltabound}, we find that
\begin{align*}
    \sup_{t \in [0,T]}\frac{1}{m}\sum_{i=1}^{m\bar{\ap}^m_j(t)} 1_{\{\ser_i^j \in (p_j\bar{G}^m(\iinta_i^j/m,t),p_j\bar{G}^m(\iinta_i^j/m,t)+\delta]\}}&
    \leq \sum_{k=0}^{  K_{\delta}-1 } \frac{N_k^{\delta,m}}{m}\sup_{a\geq 0}(X_k^{\delta,m}(a+2\delta)-X_k^{\delta,m}(a))\\
    &\leq \sum_{k=0}^{ K_{\delta}-1} \frac{N_k^{\delta,m}}{m}\sup_{a\geq 0}(X_k^{\delta}(a+2\delta)-X_k^{\delta}(a))\\
    &+\sum_{k=0}^{  K_{\delta}-1} \frac{N_k^{\delta,m}}{m}2\sup_{0 \leq k\leq   K_{\delta}-1 } \sup_{a\geq 0} |X_k^{\delta,m}(a)-X_k^{\delta}(a)|\\
    & \leq \bar{\ap}_j^m(T)w_{\pd_j,x}(2 \delta)
    \\&+\bar{\ap}_j^m(T)2\sup_{0 \leq k\leq   K_{\delta}-1 } \sup_{a\geq 0} |X_k^{\delta,m}(a)-X_k^{\delta}(a)|
\end{align*}
where $w_{\pd_j,x}$ is the modulus of continuity of the marginal cdf of $\pd_j$ in the $x$ direction.
We conclude that
\begin{align*}
&P\left(\sup_{t \in [0,T]}\frac{1}{m}\sum_{i=1}^{m\bar{\ap}^m_j(t)}1_{\{\ser_i^j \in (p_j\bar{G}^m(\iinta_i^j/m,t),p_j\bar{G}^m(\iinta_i^j/m,t)+\delta]}\geq \epsilon/4\right)\\&
\leq P(2\ar_jT w_{\pd_j,x}(2\delta)\geq\epsilon/16)\\&+P(4\ar_jT \sup_{0 \leq k\leq   K_{\delta}-1 } \sup_{a\geq 0} |X_k^{\delta,m}(a)-X_k^{\delta}(a)| \geq \epsilon/16)\\&+P(\bar{\ap}_j^m(T)\geq 2\ar_j T)
\end{align*}
taking the limsup of both sides, we find that that the limsup of the left hand side will be zero so long as $\delta$ is sufficiently small that $2\ar_jT w_{\pd_j,x}(2\delta)<\epsilon/16.$
The proof for the third term in \eqref{controlledoscdifference} is identical once one observes that for each $k = 0,1,2,...,K_{\delta}-1,$ fixed $t \in [0,T],$ and each arrival time $\iinta_i^{j,m}/m\in I_k^{\delta}$ the intervals $[t-\iinta_i^{j,m}/m,t-\iinta_i^{j,m}/m+\delta)$ are all contained in one common interval of length $2 \delta.$
The same empirical CDF argument, now applied to the
$y$-marginal of $\vartheta_j$, proves that the third term in
\eqref{controlledoscdifference} is zero.

\end{proof}
\begin{prop}
\label{fluidtightnessprop}
    $\{\bar{\boldsymbol{\ssp}}^m(\cdot)\}_{m=1}^{\infty}$ is C-tight in $D([0,\infty), \M^J).$
\end{prop}
\begin{proof}
    By Jakubowski's Criterion (See, e.g., Theorem 3.1 of \cite{jakubowski}), it suffices to prove:

\begin{enumerate}[label=(\roman*)]
\item For each $j \in \J,$ $\delta>0$, $T>0$ there exists a compact set $C_{\delta,T}$ in $\M$ such that $$\inf_{m\in\N} P\{\bar{\ssp}_j^m(t) \in C_{\delta,T}\hspace{2mm}\forall t\in [0,T] \} \geq 1-\delta, \text{ and}$$ \label{thirdjakubowski}
\item \label{withoutcont}For each $f \in \mathbf{C}^1_b(\R_+^2)$, $j \in \J,$ $\{\langle f, \bar{\ssp}^m_j(\cdot)\rangle\}_{m=1}^{\infty}$ is C-tight.\label{fourthjakubowski}
\end{enumerate}
We begin with \ref{thirdjakubowski}, which we will obtain by bounding the first moment and total mass, uniformly on $[0,T],$ by a sequence of random variables that are tight.
Examining \eqref{statespacedescriptordefequation}, we see that
$\sup_{[0,T]}\bar{\tm}_j^m(t) \leq \bar{\tm}_{0,j}^m+ \bar{\ap}^m_j(T),$ $\sup_{[0,T]}\langle \pi_x, \bar{\ssp}_j^m(t)\rangle  \leq \langle \pi_x,\bar{\ssp}_{0,j}^m\rangle+ \frac{1}{m}\sum_{i=1}^{m\bar{\ap}^m_j(T)}\ser_i^j,$ and $\sup_{[0,T]}\langle \pi_y, \bar{\ssp}^m_j(t)\rangle \leq \langle \pi_y,\bar{\ssp}_{0,j}^m\rangle+ \frac{1}{m}\sum_{i=1}^{m\bar{\ap}^m_j(T)}\pat_i^j.$
Applying the law of large numbers, we see that $\frac{1}{m}\sum_{i=1}^{m\bar{\ap}^m_j(T)}\ser_i^j \Rightarrow \ar_j \langle\pi_x, \pd_j\rangle T$ and $\frac{1}{m}\sum_{i=1}^{m\bar{\ap}^m_j(T)}\pat_i^j \Rightarrow \ar_j \langle \pi_y, \pd_j\rangle T.$
It then follows from Assumption \ref{initialconditionsassumption} that $\sup_{[0,T]}(\bar{\tm}^m_j(t)\vee \langle \pi_x, \bar{\ssp}^m_j(t)\rangle \vee \langle \pi_y, \bar{\ssp}^m_j(t)\rangle) $ is tight, and so for each $\delta >0$ there exists $K_{\delta}$ such that 
\begin{equation}
    \inf_{m\in\N} P\{\sup_{t \in [0,T]}(\bar{\tm}^m_j(t)\vee \langle \pi_x, \bar{\ssp}^m_j(t)\rangle \vee \langle \pi_y, \bar{\ssp}^m_j(t)\rangle) \leq K_{\delta}-1 \} \geq 1-\delta. \label{compactcontainmentfluid}
\end{equation}
Taking $C_{\delta,T} = \{\xi \in \M: \langle 1, \xi\rangle \vee \langle \pi_x, \xi\rangle \vee \langle \pi_y, \xi\rangle \in [0,K_{\delta}]\},$ which is a compact set in $\M,$ \ref{thirdjakubowski} is proved.

For \ref{fourthjakubowski}, fix $f \in C_b^1(\R_+^2).$
Compact containment of the sequence $\{\langle f, \bar{\ssp}^m_j(\cdot)\rangle\}_{m=1}^{\infty}$ follows immediately from \eqref{compactcontainmentfluid} and the fact that $|\langle f, \bar{\ssp}^m_j(\cdot)\rangle| \leq ||f|| \bar{\tm}^m_j(\cdot).$
To prove the controlled oscillations condition, we need to show that,
for each $T,\epsilon > 0,$
\begin{equation}
    \lim_{\delta \rightarrow 0}\liminf_{m \rightarrow \infty}P\{\sup_{t \in [0,T]}\sup_{h \in [0, \delta \wedge (T-t) ]}|\langle f, \bar{\ssp}^m_j(t+h)\rangle - \langle f, \bar{\ssp}^m_j(t)\rangle |\leq \epsilon\}=1 \label{controlled oscillationsproperty}
\end{equation}
However, we note that by \eqref{masstransportequationdef} and the fact that the service rate per job is at most $1,$ we have that for $j \in \J,t,\delta >0$
\begin{align*}
    \sup_{h \in [0,\delta]}|\langle f, \bar{\ssp}^m_j(t+h)\rangle - \langle f, \bar{\ssp}^m_j(t)\rangle |&\leq ||f|| (\bar{\ssp}_j^m(t)(B_{\delta})+\bar{\ap}_j^m(t+\delta)-\bar{\ap}^m_j(t))
    \\& +(||f_1||+||f_2||) \bar{\tm}_j^m(t) \delta,
    \numberthis \label{controlledoscdifference2}
\end{align*}
where $B_{\delta}$ is as defined in Lemma \ref{bdeltalemma}.
Therefore, the result follows from C-tightness of $\bar{\ap}^m_j(\cdot)$, Lemma \ref{bdeltalemma} and compact containment of $\bar{\tm}^m_j(\cdot)$.
\end{proof}
Now, we are almost ready to prove Theorem \ref{thm: fluid limit}. However, it will first be helpful to prove one more Lemma.
    \begin{lem}
    If $\bar{\boldsymbol{\ssp}}(\cdot)$ is a subsequential limit of $\{\bar{\boldsymbol{\ssp}}^m(\cdot)\}_{m=1}^{\infty},$ then, for each $j \in \J,$ $\langle \pi_x, \bar{\ssp}_j^m(\cdot) \rangle\Rightarrow \langle\pi_x, \bar{\ssp}_j(\cdot) \rangle $.
    \label{lem: pix convergence}
\end{lem}
\begin{proof}
Fix $j \in \J.$
By a slight abuse of notation, we will continue to denote the convergent sequence as $\{\bar{\boldsymbol{\ssp}}^m(\cdot)\}_{m=1}^{\infty}.$ 
    Applying \eqref{masstransportequationdef}, we find that, for each $\epsilon >0$ such that
    \begin{align*}
        \langle x^{1+\epsilon}, \bar{\ssp}_j^m(\cdot) \rangle \leq R_j^m(\cdot):= \langle x^{1+\epsilon}, \bar{\ssp}_j^m(0) \rangle+\frac{1}{m}\sum_{i=1}^{m \bar{\ap}^m_j(\cdot)} (\ser_i^j)^{1+\epsilon}.
    \end{align*}
    Applying the functional law of large numbers and using Assumptions \ref{assumptions} \ref{fllnassumption2} and \ref{initialconditionsassumption}, we see that there exists $\epsilon >0$ such that $R_j^m(\cdot)$ converges to $R_j(\cdot):=\langle x^{1+\epsilon}, \bar{\ssp}_j(0) \rangle+\ar_j E[(\ser_i^j)^{1+\epsilon}](\cdot).$
    Taking a Skorokhod representation, we may assume that $(\bar{{\ssp}}^m_j(\cdot),R_j^m(\cdot))\rightarrow (\bar{{\ssp}}_j(\cdot), R_j(\cdot))$
    almost surely.
    Thus, we see that, defining
    $C_T:=\sup_{m\geq 1}R_j^m(T),$
    we have that
    $$\sup_{m\geq 0}\sup_{t \in [0,T]} \langle x^{1+\epsilon}, \bar{\ssp}^m_j(t)\rangle \leq C_T < \infty.$$
  It follows that $\{\bar{\ssp}^m_j(t),\bar{\ssp}_j(t):t \in [0,T],m \geq 1\}$ is a uniformly integrable family of measures.
    Choosing a realization $\omega$ on the almost sure set on which convergence occurs, we see that, defining $\pi_{x,K}:= x \wedge K$ for each $K >0,$ we have
    \begin{align*}
        \sup_{t \in [0,T]}|\langle \pi_x, \bar{\ssp}^m_j(t)\rangle -\langle \pi_x, \bar{\ssp}_j(t)\rangle|& \leq \sup_{t \in [0,T]}|\langle \pi_x, \bar{\ssp}^m_j(t)\rangle -\langle \pi_{x,K}, \bar{\ssp}_j^m(t)\rangle|\\
        &+\sup_{t \in [0,T]}|\langle \pi_{x,K}, \bar{\ssp}^m_j(t)\rangle -\langle \pi_{x,K}, \bar{\ssp}_j(t)\rangle|\\
        &+\sup_{t \in [0,T]}|\langle \pi_{x,K}, \bar{\ssp}_j(t)\rangle -\langle \pi_{x}, \bar{\ssp}_j(t)\rangle|\\
        & \leq 2\frac{C_T}{K^{\epsilon}} + \sup_{t \in [0,T]}|\langle \pi_{x,K}, \bar{\ssp}^m_j(t)\rangle -\langle \pi_{x,K}, \bar{\ssp}_j(t)\rangle|.
    \end{align*}
    Taking the limit first in $m$ and then in $K,$ we find that $\langle \pi_x, \bar{\ssp}^m(\cdot)\rangle\rightarrow \langle \pi_x, \bar{\ssp}(\cdot)\rangle$
    uniformly on the interval $[0,T].$
    \end{proof}

\begin{proof}[Proof of Theorem \ref{thm: fluid limit}]
Let $\bar{\boldsymbol{\ssp}}(\cdot)$ be a subsequential limit of $\{\bar{\boldsymbol{\ssp}}^m(\cdot)\}_{m=1}^{\infty}.$
By a slight abuse of notation, we will use $m,$ rather than $m_k,$ to index the converging subsequence.
Condition (1) follows from Assumption \ref{assumptions} \ref{initialconditionsassumption}.
Condition (2) follows from Lemma \ref{bdeltalemma}.
In particular, for each $\delta >0,$ let $f_{\delta}$ be a bounded continuous function such that $1_{B_{\delta}}\leq f_{\delta} \leq 1_{B_{2 \delta}}.$
Then applying Lemma \ref{bdeltalemma} and the Portmanteau theorem,
\begin{align*}
   &1=\lim_{\delta\rightarrow 0} \liminf_{m\rightarrow \infty } P\{\sup_{t\in [0,T]}\bar{\ssp}_j^m(t)(B_{2\delta})\leq \epsilon\}\\
   & \leq \lim_{\delta\rightarrow 0} \liminf_{m\rightarrow \infty } P\{\sup_{t\in [0,T]}\langle f_{\delta}, \bar{\ssp}_j^m(t) \rangle \leq \epsilon\} \\
   & \leq \lim_{\delta\rightarrow 0}  P\{\sup_{t\in [0,T]}\langle f_{\delta}, \bar{\ssp}_j(t) \rangle \leq \epsilon\}\\
   &\leq \lim_{\delta\rightarrow 0}P\{\sup_{t\in [0,T]}\langle 1_{B_\delta}, \bar{\ssp}_j(t) \rangle \leq \epsilon\}\\
   & \leq P\{\sup_{t\in [0,T]} \bar{\ssp}_j(t)(\partial \R_+^2)  \leq \epsilon\}.
\end{align*}
Since $\epsilon$ can be arbitrarily small, we see that $\sup_{t\in [0,T]}  \bar{\ssp}_j(t)(\partial \R_+^2)=0$ almost surely.
To check condition (3), applying \eqref{masstransportequationdef}, we see that the workload process for class $j,$ $\langle \pi_x, \bar{\ssp}^m_j(\cdot)\rangle,$ is differentiable and decreases at rate $\frac{p_j\bar{\tm}_j^m(s)}{\sum_{i=1}^J p_i\bar{\tm}_i^m(s)}$ on intervals between arrival times, service completion times, and reneging times where there are jobs in the $j$th queue.
At arrival times for class $j,$ the workload increases by $v_i^j/m.$
At reneging times, it decreases by the work remaining on the job that reneged.
Therefore, we may write a difference equation for the workload process $\langle \pi_x, \bar{\ssp}^m_j(\cdot)\rangle:$
\begin{align*}
    \langle \pi_x, \bar{\ssp}^m_j(t)\rangle &= \langle \pi_x, \bar{\ssp}^m_j(s)\rangle- \int_s^t 1_{\{\bar{\boldsymbol{\tm}}^m(r)\neq \boldsymbol{0}\}}\frac{p_j \bar{\tm}^m_j(r)}{\sum_{i=1}^Jp_i \bar{\tm}^m_i(r)}dr + \frac{1}{m}\sum_{i=m\bar{\ap}_j^m(s)+1}^{m\bar{\ap}_j^m(t)}\ser_i^j\\&-\frac{1}{m}\sum_{i=1}^{m\bar{\ap}_j^m(t)}(\ser_i^j-p_j\bar{G}^m(\pat_i^j+\iinta_i^{j,m}/m)+p_j\bar{G}^m(\iinta_i^{j,m}/m))^+1_{\{s<\pat_i^j+\iinta_i^{j,m}/m\leq t\}}
    \\&-\frac{1}{m}\sum_{i=1}^{m\bar{\tm}_j^m(0)}(\tilde{\ser}_i^j-p_j\bar{G}^m(\tilde{\pat}_i^j))^+1_{\{\tilde{\pat}_i^j\in (s,t]\}}
    \numberthis
    \label{workloaddiffeq}
\end{align*}
We note also that, applying \eqref{masstransportequationdef} directly, we have that
\begin{align*}
    \langle \pi_x, \bar{\ssp}^m_j(s)\rangle &= \frac{1}{m}\sum_{i=1}^{m\bar{\ap}_j^m(s)}(\ser_i^j-p_j\bar{G}^m(s)+p_j\bar{G}^m(\iinta_i^{j,m}/m))^+1_{\{s<\pat_i^j+\iinta_i^{j,m}/m\}}
    \\&+\frac{1}{m}\sum_{i=1}^{m\bar{\tm}_j^m(0)}(\tilde{\ser}_i^j-p_j\bar{G}^m(s))^+1_{\{\tilde{\pat}_i^j>s\}}
    \numberthis
    \label{workloaddiffeq2}
\end{align*}
Combining \eqref{workloaddiffeq} and \eqref{workloaddiffeq2}, we see that 
\begin{align*}
     \sum_{j=1}^J \langle \pi_x,\bar{\ssp}^m_j(t)\rangle &\geq - (t-s)+\frac{1}{m}\sum_{j=1}^J\sum_{i=m\bar{\ap}_j^m(s)+1}^{m\bar{\ap}_j^m(t)}\ser_i^j1_{\{\pat_i^j+\iinta_i^{j,m}/m> t\}}
     \numberthis
    \label{lowerboundoffluidmodel}
\end{align*}
We now claim that the right hand side of the inequality above converges in distribution to
\begin{equation}
 - (t-s)+\sum_{j=1}^J\int_s^t \ar_j \langle \pi_x 1_{(0,\infty)\times (t-r,\infty)},\pd_j\rangle dr.
\end{equation}
Indeed, the convergence of the second term to the desired form follows from the functional law of large numbers combined with a real analysis argument
For more details on the proof of convergence to the above expression, see the proof of Lemma 8.2 in \cite{loeserwilliams}.
Noticing that the load parameter, which we have assumed in Assumptions \ref{assumptions} is greater than $1,$ can be re-written
$$1<\rho = \sum_{j=1}^J \ar_j \langle \pi_x,\pd_j\rangle . $$
By taking a Skorokhod representation that includes both sides of \eqref{lowerboundoffluidmodel} as well as each side evaluated at each rational pair $(s,t) \in [0,T],$ we may say that, almost surely, for all $0 \leq s \leq t$
\begin{align*}
   \sum_{j=1}^J\langle \pi_x,\bar{\ssp}_j(t)\rangle&\geq   (\rho-1)(t-s)-\int_s^t \sum_{j=1}^J \ar_j \langle \pi_x1_{(0,\infty)\times [0,t-r]},\pd_j\rangle dr
\end{align*}
By an abuse of notation, we denote all quantities with the same variable names.
Because we are ultimately only interested in the distributional properties of the limit $\bar{\ssp}_j(\cdot),$ this step is valid.
Then, using Assumptions \ref{assumptions} \ref{basicassumptions}, particularly the fact that $\pd_j([0,\infty)\times \{0\})=0$ and $\langle \pi_x, \pd_j\rangle < \infty,$ we choose $\epsilon$ such that $ \sum_{j=1}^J \ar_j \langle \pi_x 1_{(0,\infty)\times [0,\epsilon]},\pd_j\rangle \leq \frac{\rho-1}{2}$, we see that for $t\leq s+\epsilon$,
\begin{equation}
    \sum_{j=1}^J\langle \pi_x,\bar{\ssp}_j(t)\rangle\geq  (\rho-1)(t-s)/2.
\end{equation}
It follows that $\bar{\boldsymbol{\ssp}}(\cdot)$ is nonzero on $(s,s+\epsilon]$.
Because $s$ is arbitrary, this is sufficient to check the third condition.

In order to check (4), one can simply take the limit of \eqref{differentialmasstransport} term-by-term using $g(t,\cdot,\cdot):= f(\cdot,\cdot).$
The first term on the right hand side converges by Assumption \ref{assumptions} \ref{initialconditionsassumption}. The second two terms converge uniformly on compacts by a standard real analysis argument using a Skorokhod representation, continuity of the derivatives of $f,$ the fact that $1_{\{\bar{\boldsymbol{\tm}}^m(s) \neq \boldsymbol{0}\}} \frac{p_j}{\wmass(\bar{\boldsymbol{\tm}}^m(s))}$ converges to a continuous function because $\wmass(\bar{\boldsymbol{\tm}}^m(s))$ is eventually bounded away from $0$, and bounded convergence (for a detailed version of this argument, see the proof of Lemma 9.5 in \cite{loeserwilliams}).
The last term converges to $\ar_j t\langle f, \pd_j\rangle$ by the functional law of large numbers (see, e.g., Lemma A.2 in \cite{gromollpuhawilliams}).

\end{proof}
\section{Proofs of Diffusion Limit Results}
\label{sec: diffusion proofs}
For the remainder of this section, we will assume that we are working with a sequence of diffusion-scaled models $\{\hat{\ssp}^m_j(\cdot) \}_{m=1}^{\infty}$ that satisfy Assumption \ref{assumptions}. The further assumption that for $j \in \J,$ $\fl_{0,j},\pd_j$ have densities $h_j,g_j \in C_b^2(\R_+^2)$, respectively, and CDFs in $C_b^1(\R_+^2)$ as well will be assumed from this point forward.
We would like to use the results in \cite{loeser2025diffusionlimitsmeasurevaluedqueueing}, so we decompose \eqref{diffusiondefinitioneqn} into a sequence of good renewal driven systems.
We use the martingale decomposition in \cite{loeser2025diffusionlimitsmeasurevaluedqueueing}, Proposition 4.1, in order to decompose the term $$\sum_{i=1}^{\ap^m_j(t)} \varphi(\iinta_i^{j,m},\ser_i^{j,m},\pat_i^{j,m})$$ of \eqref{differentialmasstransport}.
Then, for a bounded measurable function
$\varphi:\R_+^3\rightarrow\R$, define the averaged term
\begin{align}
\avg_{\varphi}^{\ap_j,j,m}(t)
&:=
\int_0^t
\left\langle
\varphi(s,\cdot,\cdot),\pd_j
\right\rangle
d\ap_j^m(s),
\end{align}
and the martingale term
\begin{align}
\mart_{\varphi}^{\ap_j,j,m}(t)
&:=
\sum_{i=1}^{\ap_j^m(t)}
\bigg(
\varphi(\iinta_i^{j,m},\ser_i^{j,m},\pat_i^{j,m})
-
\left\langle
\varphi(\iinta_i^{j,m},\cdot,\cdot),\pd_j
\right\rangle
\bigg).
\end{align}
We will also use the following state-dependent extension. For a
bounded measurable function
$\psi:\R_+^4\rightarrow\R$, define
\begin{align}
\othermart_{\psi}^{\ap_j,j,m}(t)
:=
\sum_{i=1}^{\ap_j^m(t)}
\bigg(
&
\psi\left(
\iinta_i^{j,m},
\ser_i^{j,m},
\pat_i^{j,m},
G^m(\iinta_i^{j,m}-)
\right)
\\
&-
\left\langle
\psi\left(
\iinta_i^{j,m},
\cdot,
\cdot,
G^m(\iinta_i^{j,m}-)
\right),
\pd_j
\right\rangle
\bigg).
\end{align}
By Proposition~4.1 of
\cite{loeser2025diffusionlimitsmeasurevaluedqueueing},
$\mart_{\varphi}^{\ap_j,j,m}(\cdot)$ is a martingale whenever
$\varphi:\R_+^3\to\R$ is bounded and measurable. The same argument
shows that
$\othermart_{\psi}^{\ap_j,j,m}(\cdot)$ is a martingale whenever
$\psi:\R_+^4\to\R$ is bounded and measurable.

Following the steps in \cite{loeser2025diffusionlimitsmeasurevaluedqueueing} \S 4.1 for the equations \eqref{differentialmasstransport} and \eqref{fluidtransportequationdifferential} let $f \in C_b^1(\R_+^3)$ such that, for each $t\geq 0,$ the function $f(t,\cdot,\cdot):\R_+^2\rightarrow \R$ as well as the first partials $f_2,f_3$ are zero on $\partial \R_+^2.$ Then almost surely for $t \geq 0, j \in \J$
\begin{align*}
    \langle f(t,\cdot,\cdot),\hat{\ssp}_j^m(t) \rangle &= \langle f(0,\cdot,\cdot),\hat{\ssp}_j^m(0) \rangle+\int_0^t\langle f_1(s,\cdot,\cdot), \hat{\ssp}^m_j(s) \rangle ds- \int_0^t \langle f_3(s,\cdot,\cdot),\hat{\ssp}_j^m(s)\rangle ds\\&- \int_0^tp_j\left(\frac{\langle f_2(s,\cdot,\cdot), \hat{\ssp}_j^m(s) \rangle}{\wmass(\boldsymbol{\flm}(s))}-\frac{\langle f_2(s,\cdot,\cdot), \bar{\ssp}_j^m(s) \rangle}{\wmass(\bar{\boldsymbol{\tm}}^m(s))\wmass(\boldsymbol{\flm}(s))} \wmass(\hat{\boldsymbol{\tm}}^m(s)) \right)  ds
    \\&+\int_0^t \langle f(s,\cdot,\cdot), \pd_j \rangle d\hat{\ap}^m_j(s) +\hat{\mart}^{\ap_j,j,m}_f(t).
    \numberthis
    \label{renewaldrivenequation}
\end{align*}
where the diffusion-scaling of
${\mart}^{\ap_j,j,m}_{\varphi}(t)$ is
\begin{align*}
    \hat{\mart}^{\ap_j,j,m}_{\varphi}(t)&:=\frac{1}{\sqrt{m}}\sum_{i=1}^{m\bar{\ap}_j^m(t)}\bigg(\varphi(\iinta_i^{j,m}/m,\ser_i^{j,m},\pat_i^{j,m})\\&\hspace{10mm}- \langle \varphi(\iinta_i^{j,m}/m,\cdot,\cdot),\pd_j\rangle\bigg),
    \numberthis
    \label{eq: diffusion scaled Y}
\end{align*}

and the diffusion-scaling of
$\othermart^{\ap_j,j,m}_{\psi}(t)$ is
\begin{align*}
    \hat{\othermart}^{\ap_j,j,m}_{\psi}(t)&:=\frac{1}{\sqrt{m}}\sum_{i=1}^{m\bar{\ap}_j^m(t)}\bigg(\psi(\iinta_i^{j,m}/m,\ser_i^{j,m},\pat_i^{j,m},\bar{G}^m(\iinta_i^{j,m}/m-))\\&\hspace{10mm}- \langle \psi(\iinta_i^{j,m}/m,\cdot,\cdot,\bar{G}^m(\iinta_i^{j,m}/m-)),\pd_j\rangle\bigg).
    \numberthis
    \label{eq: diffusion scaled Y tilde}
\end{align*}

\begin{rem}
    We note that we have omitted the indicator term $1_{\{\bar{\boldsymbol{\tm}}^m(s) \neq \boldsymbol{0}\}}$ in the decomposition above because we may (and do) take $\wmass(\bar{\boldsymbol{\tm}}^m(\cdot))$ on the compact interval $[0,T]$ to be bounded below by a constant $C_T>0$ that depends only on the value of $\boldsymbol{\fl}_0,T,$ and the model parameters for large $m$ (see Remark 5.1 of \cite{loeser2025diffusionlimitsmeasurevaluedqueueing}). More precisely, for each $T>0$ there exist events $\Omega_T^m$
with $P(\Omega_T^m)\to 1$ on which
$\wmass(\bar{\boldsymbol{\tm}}^m(\cdot))$ is uniformly bounded away
from zero on $[0,T]$. Thus, throughout the arguments below, equations
in which the corresponding nonemptiness indicators have been omitted
are understood on $\Omega_T^m$; this does not affect any convergence
in distribution result.
    \label{nonzerorem}
\end{rem}

Alternatively, one can compare \eqref{masstransportequationdef} and \eqref{fluidtransporteqn} and do the same decomposition.
However, this will require us to develop a mass-transported representation of the
martingale term, which we will do through the state-dependent martingale field
\[
\left\{
\hat{\othermart}^{\ap_j,j,m}_{g_{u,s}}(\cdot):
(u,s)\in\R_+^2
\right\}.
\]
   Because this is key, and a little bit delicate, we spend some time now defining this mass-transported field, which we will denote $ \{\Phi_{f,t}^{j,m}(r): 0\leq r \leq t, m \in \N\}$.
   We note that, because the mass-transport will be evaluated at a random endpoint $(p_j\bar{G}^m(t),t)$, we are not asserting the martingale property for $ \Phi_{f,t}^{j,m}(r).$ 
Rather, we will use the martingale property of each $\hat{\othermart}^{\ap_j,j,m}_{g_{u,s}}(\cdot),$ a C-tightness argument, and the fact that $\bar{G}^m(\cdot)$ is asymptotically deterministic in order to characterize our limiting martingale field.
In particular, almost surely, for $f \in \mathbf{C}_b^1(\R_+^2),$ $t \geq 0,$ 
\begin{align*}
    \langle f,\hat{\ssp}^m_j(t)\rangle & = U_f^{j,m}(t)  +\sqrt{m}\langle t_{p_j\bar{G}^m(t),t}f-t_{p_j\bar{G}(t),t}f, \fl_{0,j}\rangle 
     \\&+\int_0^t\langle \sqrt{m}(t_{p_j\bar{G}^m(s,t),t-s}f-t_{p_j\bar{G}(s,t),t-s}f), \pd_j\rangle d\bar{\ap}^m_j(s)
    \numberthis ,\label{mainMeq}
\end{align*}
where, defining, for $0 \leq r \leq t,$ 
\begin{align*}
    \Phi_{f,t}^{j,m}(r)&:=\hat{\othermart}^{\ap_j,j,m}_{g_{p_j\bar{G}^m(t),t}}(r),
\end{align*}
and we are evaluating a function $g_{u,s}(w,x,y,z):=t_{(u-p_jz)^+,(s-w)^+}f(x,y)$ at the (time, incoming noise, system state) vector $(\iinta_i^{j,m}/m, \ser_i^{j,m},\pat_i^{j,m},\bar{G}^m(\iinta_i^{j,m}/m-))$ with a random spatial shift $(u,s)=(p_j\bar{G}^m(t),t).$
\begin{align*}
    U_f^{j,m}(t) &= F_{f}^{j,c,m}(p_j\bar{G}^m(t),t) +\Phi_{f,t}^{j,m}(t)
     +\int_0^t\langle t_{p_j\bar{G}(s,t),t-s}f, \pd_j \rangle d\hat{\ap}_j^m(s).
     \numberthis \label{Udef}
    \end{align*}
Here, $F_{f}^{j,c,m}$ is as given in Assumptions \ref{assumptions} \ref{initialconditionsassumption}.


The rest of this section will proceed as follows. Examining equation \eqref{mainMeq}, we notice that it looks quite similar to the limiting diffusion equation \eqref{rjlimitdef}.
Thus, to prove theorem \ref{lhatconvergencethm}, we will use a compactness-uniqueness type argument.
We will first reduce tightness of
$\langle f,\hat{\ssp}_j^m(\cdot)\rangle$, for a given
$f\in\mathscr S$ and $j\in\J$, to tightness of
$\{U_f^{j,m}(\cdot)\}_{m=1}^{\infty}$.
We will then prove convergence of $\{U_f^{j,m}(\cdot)\}_{m=1}^{\infty}$ to the process given in Definition \ref{Udefdef}, which proves Theorem \ref{tightnessresult}.
We then prove convergence of the right hand side of \eqref{mainMeq} to the right hand side of \eqref{rjlimitdef} term-by-term.
The convergence result (Theorem \ref{lhatconvergencethm}) then follows from uniqueness of solutions for mass transport diffusion model solutions (Theorem \ref{Llimit}), which is proved next.
\subsection{Proof of Tightness}
Before proving tightness, we prove the following Lemma.
It is purely computational, but it is the key computation that makes the subsequent bounds clear.

\begin{lem}
\label{lem: weird calculation}
    Let $\mu\in\boldsymbol{M}$, and define
    \[
        F^\mu(x,t)
        :=
        \mu\bigl([0,x]\times(t,\infty)\bigr),
        \qquad x,t\geq0.
    \]
    Suppose that
    $F^\mu\in C_b^2(\R_+^2)$. Let $a\geq0$, and let
    $H_1,H_2:[a,\infty)\to\R_+$ be increasing functions.
    Then, for every $a\leq r<s$ and $t\geq0$,
    \begin{align*}
        &
        \Big|
        \left\langle
        1_{(H_1(s),\infty)\times(t,\infty)},\mu
        \right\rangle
        -
        \left\langle
        1_{(H_2(s),\infty)\times(t,\infty)},\mu
        \right\rangle
        \\
        &\qquad
        -
        \left(
        \left\langle
        1_{(H_1(r),\infty)\times(t,\infty)},\mu
        \right\rangle
        -
        \left\langle
        1_{(H_2(r),\infty)\times(t,\infty)},\mu
        \right\rangle
        \right)
        \Big|
        \\
        &\leq
        \left|
        H_1(s)-H_1(r)
        -
        \bigl(H_2(s)-H_2(r)\bigr)
        \right|
        \left|F_1^\mu(H_1(r),t)\right|
        \\
        &\quad+
        \left\|F_{11}^\mu\right\|_\infty
        |H_1(r)-H_2(r)|^2
        \\
        &\quad+
        \left\|F_{11}^\mu\right\|_\infty
        |H_1(s)-H_2(s)|
        \left(
        H_1(s)-H_1(r)
        +
        |H_1(s)-H_2(s)|
        \right).
    \end{align*}
\end{lem}

\begin{proof}
    The proof is a simple calculation. For $u\geq a$, we have
    \begin{align*}
        &
        \left\langle
        1_{(H_1(u),\infty)\times(t,\infty)},\mu
        \right\rangle
        -
        \left\langle
        1_{(H_2(u),\infty)\times(t,\infty)},\mu
        \right\rangle
        \\
        &=
        F^\mu(H_2(u),t)-F^\mu(H_1(u),t)
        \\
        &=
        sgn\bigl(H_2(u)-H_1(u)\bigr)
        \Big(
        F^\mu\bigl(\max\{H_1(u),H_2(u)\},t\bigr)
        \\
        &\hspace{45mm}
        -
        F^\mu\bigl(\min\{H_1(u),H_2(u)\},t\bigr)
        \Big).
    \end{align*}
    Therefore, by the mean value theorem, there exist
    $\xi(r)$ and $\xi(s)$ such that
    \[
        H_1(r)+\xi(r)
        \in
        \left[
        \min\{H_1(r),H_2(r)\},
        \max\{H_1(r),H_2(r)\}
        \right]
    \]
    and
    \[
        H_1(s)+\xi(s)
        \in
        \left[
        \min\{H_1(s),H_2(s)\},
        \max\{H_1(s),H_2(s)\}
        \right],
    \]
    and the expression on the left-hand side of the desired
    inequality equals
    \begin{align*}
        \Big|
        &(H_2(s)-H_1(s))
        F_1^\mu(H_1(s)+\xi(s),t)
        \\
        &-
        (H_2(r)-H_1(r))
        F_1^\mu(H_1(r)+\xi(r),t)
        \Big|.
    \end{align*}
    Adding and subtracting the appropriate terms and applying the
    triangle inequality, this is bounded above by
    \begin{align*}
        &
        \left|
        (H_2(s)-H_1(s))
        -
        (H_2(r)-H_1(r))
        \right|
        \left|F_1^\mu(H_1(r),t)\right|
        \\
        &\quad+
        |H_2(r)-H_1(r)|
        \left|
        F_1^\mu(H_1(r)+\xi(r),t)
        -
        F_1^\mu(H_1(r),t)
        \right|
        \\
        &\quad+
        |H_2(s)-H_1(s)|
        \left|
        F_1^\mu(H_1(s)+\xi(s),t)
        -
        F_1^\mu(H_1(r),t)
        \right|.
    \end{align*}
    By the choice of $\xi(r)$ and $\xi(s)$,
    \[
        |\xi(r)|
        \leq
        |H_1(r)-H_2(r)|
    \]
    and, since $H_1$ is increasing, by the triangle inequality
    \[
        |H_1(s)+\xi(s)-H_1(r)|
        \leq
        H_1(s)-H_1(r)
        +
        |H_1(s)-H_2(s)|.
    \]
    Since $F_1^\mu(\cdot,t)$ is Lipschitz with Lipschitz
    constant $\|F_{11}^\mu\|_\infty$, it follows that
    \begin{align*}
        &
        \left|
        F_1^\mu(H_1(r)+\xi(r),t)
        -
        F_1^\mu(H_1(r),t)
        \right|
        \\
        &\leq
        \left\|F_{11}^\mu\right\|_\infty
        |H_1(r)-H_2(r)|
    \end{align*}
    and
    \begin{align*}
        &
        \left|
        F_1^\mu(H_1(s)+\xi(s),t)
        -
        F_1^\mu(H_1(r),t)
        \right|
        \\
        &\leq
        \left\|F_{11}^\mu\right\|_\infty
        \left(
        H_1(s)-H_1(r)
        +
        |H_1(s)-H_2(s)|
        \right).
    \end{align*}
    Substituting these bounds into the preceding inequality
    completes the proof.
\end{proof}

\begin{lem}
\label{componenttightnesslemtm}
    Let
    $\{\hat{\boldsymbol{\ssp}}^m(\cdot)\}_{m=1}^{\infty}$
    be a sequence of diffusion-scaled models satisfying
    Assumption~\ref{assumptions}. Suppose that, for every
    $j\in\J$,
    \[
        \{U_1^{j,m}(\cdot)\}_{m=1}^{\infty}
    \]
    is C-tight. Then
    \[
        \left\{
        \left(
        \wmass(\hat{\boldsymbol{\tm}}^m(\cdot)),
        \hat G^m(\cdot)
        \right)
        \right\}_{m=1}^{\infty}
    \]
    is C-tight in $D([0,\infty),\R^2)$.
\end{lem}

\begin{proof}
    Fix $T>0$. We begin by proving compact containment of
    \[
        \left\{
        \wmass(\hat{\boldsymbol{\tm}}^m(\cdot))
        \right\}_{m=1}^{\infty}.
    \]
    For $0\leq r\leq t$, write
    \[
        \bar G^m(r,t):=\bar G^m(t)-\bar G^m(r),
        \qquad
        \bar G(r,t):=\bar G(t)-\bar G(r),
    \]
    and
    \[
        \hat G^m(r,t):=\hat G^m(t)-\hat G^m(r).
    \]
    For $t\geq0$, define
    \begin{align*}
        I_j^m(t)
        &:=
        \sqrt{m}
        \left\langle
        1_{(p_j\bar G^m(0,t),\infty)\times(t,\infty)}
        -
        1_{(p_j\bar G(0,t),\infty)\times(t,\infty)},
        \fl_{0,j}
        \right\rangle,
    \end{align*}
    and, for $0\leq r\leq t$, define
    \begin{align*}
        R_j^m(r,t)
        &:=
        \sqrt{m}
        \left\langle
        1_{(p_j\bar G^m(r,t),\infty)\times(t-r,\infty)}
        -
        1_{(p_j\bar G(r,t),\infty)\times(t-r,\infty)},
        \pd_j
        \right\rangle.
    \end{align*}

    Applying \eqref{mainMeq} and \eqref{Udef} with $f=1$, we
    obtain, for $t\geq0$,
    \begin{align}
        \wmass(\hat{\boldsymbol{\tm}}^m(t))
        &=
        \sum_{j=1}^J p_jU_1^{j,m}(t)
        +
        \sum_{j=1}^Jp_jI_j^m(t)
        +
        \sum_{j=1}^Jp_j
        \int_0^tR_j^m(r,t)\,d\bar{\ap}_j^m(r).
        \label{Leq}
    \end{align}

    By the mean value theorem and the boundedness of the first
    partial derivatives of the joint CDFs,
    \begin{align*}
        |I_j^m(t)|
        &\leq
        p_j
        \left\|F_1^{\fl_{0,j}}\right\|_\infty
        |\hat G^m(0,t)|,
        \numberthis   \label{Imvtbound}
        \\
        |R_j^m(r,t)|
        &\leq
        p_j
        \left\|F_1^{\pd_j}\right\|_\infty
        |\hat G^m(r,t)|.
         \numberthis \label{Rmvtbound}
    \end{align*}
    Moreover,
    \begin{align*}
        |\hat G^m(r,t)|
        &\leq
        \sqrt{m}
        \int_r^t
        \left|
        \frac{1}{\wmass(\bar{\boldsymbol{\tm}}^m(u))}
        -
        \frac{1}{\wmass(\boldsymbol{\flm}(u))}
        \right|du
        \\
        &=
        \int_r^t
        \frac{
        |\wmass(\hat{\boldsymbol{\tm}}^m(u))|
        }
        {
        \wmass(\bar{\boldsymbol{\tm}}^m(u))
        \wmass(\boldsymbol{\flm}(u))
        }
        \,du. 
    \end{align*}
    By Remark~\ref{nonzerorem}, there exist a deterministic constant
$c_T>0$ and events $\Omega_T^m$ such that
\[
    P(\Omega_T^m)\longrightarrow1
\]
and, on $\Omega_T^m$,
\[
    \inf_{u\in[0,T]}
    \wmass(\bar{\boldsymbol{\tm}}^m(u))
    \wmass(\boldsymbol{\flm}(u))
    \geq c_T.
\]
Consequently, on $\Omega_T^m$, for every
$0\leq r\leq t\leq T$,
\begin{equation}
    |\hat G^m(r,t)|
    \leq
    C_T\int_r^t
    |\wmass(\hat{\boldsymbol{\tm}}^m(u))|\,du,
    \qquad C_T:=c_T^{-1}.
    \label{eq: Ghat mass bound}
\end{equation}
    On $\Omega_T^m,$ it follows from \eqref{Leq} that
    \begin{align*}
        |\wmass(\hat{\boldsymbol{\tm}}^m(t))|
        &\leq
        \sum_{j=1}^Jp_j|U_1^{j,m}(t)|
        \\
        &\quad+
        C_T\sum_{j=1}^J
        p_j^2
        \left\|F_1^{\fl_{0,j}}\right\|_\infty
        \int_0^t
        |\wmass(\hat{\boldsymbol{\tm}}^m(u))|\,du
        \\
        &\quad+
        C_T\sum_{j=1}^J
        p_j^2
        \left\|F_1^{\pd_j}\right\|_\infty
        \int_0^t\int_r^t
        |\wmass(\hat{\boldsymbol{\tm}}^m(u))|
        \,du\,d\bar{\ap}_j^m(r),
    \end{align*}
    By Fubini's theorem,
    \begin{align*}
        \int_0^t\int_r^t
        |\wmass(\hat{\boldsymbol{\tm}}^m(u))|
        \,du\,d\bar{\ap}_j^m(r)
        &=
        \int_0^t
        |\wmass(\hat{\boldsymbol{\tm}}^m(u))|
        \left(
        \int_0^u d\bar{\ap}_j^m(r)
        \right)du
        \\
        &=
        \int_0^t
        |\wmass(\hat{\boldsymbol{\tm}}^m(u))|
        \bar{\ap}_j^m(u)\,du.
    \end{align*}
    Consequently,
    \begin{align}
        |\wmass(\hat{\boldsymbol{\tm}}^m(t))|
        &\leq
        \sum_{j=1}^Jp_j|U_1^{j,m}(t)|
        \notag\\
        &\quad+
        C_T\int_0^t
        |\wmass(\hat{\boldsymbol{\tm}}^m(u))|
        \left(
        \sum_{j=1}^J
        p_j^2
        \left\|F_1^{\fl_{0,j}}\right\|_\infty
        +
        \sum_{j=1}^J
        p_j^2
        \left\|F_1^{\pd_j}\right\|_\infty
        \bar{\ap}_j^m(u)
        \right)du,
        \label{Geq}
    \end{align}
    with high probability as $m \rightarrow \infty$.
    Applying Lemma~5.1 of
\cite{loeser2025diffusionlimitsmeasurevaluedqueueing}
pathwise on $\Omega_T^m$, together with the C-tightness of
$\{U_1^{j,m}\}_{m=1}^{\infty}$ for every $j\in\J$ and the
compact containment of
$\{\bar{\ap}_j^m\}_{m=1}^{\infty}$, gives the required compact
containment on $\Omega_T^m$. Since
$P(\Omega_T^m)\to1$, this implies compact containment of
\[
    \left\{
    \wmass(\hat{\boldsymbol{\tm}}^m(\cdot))
    \right\}_{m=1}^{\infty}.
\]
    Here we have used the renewal-process FLLN $\bar{\ap}_j^m(\cdot)\longrightarrow\ar_j(\cdot),$ uniformly on compact time intervals in probability. Recall that
    \[
        \hat G^m(t)
        :=
        \sqrt{m}\bigl(\bar G^m(t)-\bar G(t)\bigr).
    \]
    By \eqref{eq: Ghat mass bound}, for
    $0\leq s\leq t\leq T$,
    \begin{align*}
        |\hat G^m(t)-\hat G^m(s)|
        &\leq
        C_T\int_s^t
        |\wmass(\hat{\boldsymbol{\tm}}^m(u))|\,du,
    \end{align*}
    with high probability as $m\rightarrow \infty.$
    Thus, on the event
    \[
    \Omega_T^m
    \cap
    \left\{
    \sup_{u\in[0,T]}
    |\wmass(\hat{\boldsymbol{\tm}}^m(u))|
    \leq M
    \right\},
\]
    
    we have
    \[
        |\hat G^m(t)-\hat G^m(s)|
        \leq C_TM|t-s|,
        \qquad 0\leq s\leq t\leq T.
    \]
    Compact containment of
    $\wmass(\hat{\boldsymbol{\tm}}^m(\cdot))$ therefore implies
    C-tightness of $\{\hat G^m\}_{m=1}^{\infty}$.

    It remains to prove C-tightness of
    \[
        \left\{
        \wmass(\hat{\boldsymbol{\tm}}^m(\cdot))
        \right\}_{m=1}^{\infty}.
    \]
    Applying \eqref{Leq} at times $s$ and $t$, where
    $0\leq s\leq t\leq T$, gives
    \begin{align}
        &\left|
        \wmass(\hat{\boldsymbol{\tm}}^m(t))
        -
        \wmass(\hat{\boldsymbol{\tm}}^m(s))
        \right|
        \notag\\
        &\leq
        \sum_{j=1}^Jp_j
        |U_1^{j,m}(t)-U_1^{j,m}(s)|
        \notag\\
        &\quad+
        \sum_{j=1}^Jp_j
        |I_j^m(t)-I_j^m(s)|
        \notag\\
        &\quad+
        \sum_{j=1}^Jp_j
        \int_0^s
        |R_j^m(r,t)-R_j^m(r,s)|
        \,d\bar{\ap}_j^m(r)
        \notag\\
        &\quad+
        \sum_{j=1}^Jp_j
        \int_s^t
        |R_j^m(r,t)|
        \,d\bar{\ap}_j^m(r).
        \label{eq: mass oscillation}
    \end{align}

    Fix $0\leq r\leq s\leq t\leq T$. We first separate the
    change in the first coordinate from the change in the second
    coordinate:
    \begin{align*}
        |R_j^m(r,t)-R_j^m(r,s)|
        &\leq
        \sqrt{m}
        \Big|
        \left\langle
        1_{(p_j\bar G^m(r,t),\infty)\times(t-r,\infty)}
        -
        1_{(p_j\bar G(r,t),\infty)\times(t-r,\infty)},
        \pd_j
        \right\rangle
        \\
        &\qquad\qquad-
        \left\langle
        1_{(p_j\bar G^m(r,s),\infty)\times(t-r,\infty)}
        -
        1_{(p_j\bar G(r,s),\infty)\times(t-r,\infty)},
        \pd_j
        \right\rangle
        \Big|
        \\
        &\quad+
        \sqrt{m}
        \Big|
        \left\langle
        1_{(p_j\bar G^m(r,s),\infty)\times(t-r,\infty)}
        -
        1_{(p_j\bar G(r,s),\infty)\times(t-r,\infty)},
        \pd_j
        \right\rangle
        \\
        &\qquad\qquad-
        \left\langle
        1_{(p_j\bar G^m(r,s),\infty)\times(s-r,\infty)}
        -
        1_{(p_j\bar G(r,s),\infty)\times(s-r,\infty)},
        \pd_j
        \right\rangle
        \Big|.
    \end{align*}

    Applying Lemma~\ref{lem: weird calculation} to the first
    term, with $a=r$ and
    \[
        H_1(v):=p_j\bar G^m(r,v),
        \qquad
        H_2(v):=p_j\bar G(r,v),
        \qquad r\leq v,
    \]
    and with the second-coordinate threshold fixed at $t-r$, we
    obtain
    \begin{align*}
        &\sqrt{m}
        \Big|
        \left\langle
        1_{(p_j\bar G^m(r,t),\infty)\times(t-r,\infty)}
        -
        1_{(p_j\bar G(r,t),\infty)\times(t-r,\infty)},
        \pd_j
        \right\rangle
        \\
        &\qquad-
        \left\langle
        1_{(p_j\bar G^m(r,s),\infty)\times(t-r,\infty)}
        -
        1_{(p_j\bar G(r,s),\infty)\times(t-r,\infty)},
        \pd_j
        \right\rangle
        \Big|
        \\
        &\leq
        p_j|\hat G^m(s,t)|
        \left|
        F_1^{\pd_j}
        \bigl(p_j\bar G^m(r,s),t-r\bigr)
        \right|
        \\
        &\quad+
        \frac{p_j^2}{\sqrt{m}}
        \left\|F_{11}^{\pd_j}\right\|_\infty
        |\hat G^m(r,s)|^2
        \\
        &\quad+
        p_j^2
        \left\|F_{11}^{\pd_j}\right\|_\infty
        |\hat G^m(r,t)|
        \left(
        \bar G^m(s,t)
        +
        \frac{|\hat G^m(r,t)|}{\sqrt{m}}
        \right).
    \end{align*}

    The second term corresponds to a rectangle with horizontal
    side length
    \[
        p_j
        \left|
        \bar G^m(r,s)-\bar G(r,s)
        \right|
        =
        \frac{p_j|\hat G^m(r,s)|}{\sqrt{m}}
    \]
    and vertical side length $t-s$. Therefore,
    \begin{align*}
        &\sqrt{m}
        \Big|
        \left\langle
        1_{(p_j\bar G^m(r,s),\infty)\times(t-r,\infty)}
        -
        1_{(p_j\bar G(r,s),\infty)\times(t-r,\infty)},
        \pd_j
        \right\rangle
        \\
        &\qquad-
        \left\langle
        1_{(p_j\bar G^m(r,s),\infty)\times(s-r,\infty)}
        -
        1_{(p_j\bar G(r,s),\infty)\times(s-r,\infty)},
        \pd_j
        \right\rangle
        \Big|
        \\
        &\leq
        p_j
        \left\|F_{12}^{\pd_j}\right\|_\infty
        |\hat G^m(r,s)|(t-s).
    \end{align*}
    Combining these two estimates yields
    \begin{align}
        |R_j^m(r,t)-R_j^m(r,s)|
        &\leq
        p_j|\hat G^m(s,t)|
        \left|
        F_1^{\pd_j}
        \bigl(p_j\bar G^m(r,s),t-r\bigr)
        \right|
        \notag\\
        &\quad+
        \frac{p_j^2}{\sqrt{m}}
        \left\|F_{11}^{\pd_j}\right\|_\infty
        |\hat G^m(r,s)|^2
        \notag\\
        &\quad+
        p_j^2
        \left\|F_{11}^{\pd_j}\right\|_\infty
        |\hat G^m(r,t)|
        \left(
        \bar G^m(s,t)
        +
        \frac{|\hat G^m(r,t)|}{\sqrt{m}}
        \right)
        \notag\\
        &\quad+
        p_j
        \left\|F_{12}^{\pd_j}\right\|_\infty
        |\hat G^m(r,s)|(t-s).
        \label{eq: arrival strip oscillation}
    \end{align}

    The same argument, with $\fl_{0,j}$ in place of $\pd_j$ and
    $r=0$, gives
    \begin{align}
        |I_j^m(t)-I_j^m(s)|
        &\leq
        p_j|\hat G^m(s,t)|
        \left|
        F_1^{\fl_{0,j}}
        \bigl(p_j\bar G^m(0,s),t\bigr)
        \right|
        \notag\\
        &\quad+
        \frac{p_j^2}{\sqrt{m}}
        \left\|F_{11}^{\fl_{0,j}}\right\|_\infty
        |\hat G^m(0,s)|^2
        \notag\\
        &\quad+
        p_j^2
        \left\|F_{11}^{\fl_{0,j}}\right\|_\infty
        |\hat G^m(0,t)|
        \left(
        \bar G^m(s,t)
        +
        \frac{|\hat G^m(0,t)|}{\sqrt{m}}
        \right)
        \notag\\
        &\quad+
        p_j
        \left\|F_{12}^{\fl_{0,j}}\right\|_\infty
        |\hat G^m(0,s)|(t-s).
        \label{eq: initial strip oscillation}
    \end{align}
Finally, from \eqref{Rmvtbound}
    \begin{align*}
        \int_s^t|R_j^m(r,t)|\,d\bar{\ap}_j^m(r)
        &\leq
        p_j
        \left\|F_1^{\pd_j}\right\|_\infty
        \sup_{r\in[s,t]}|\hat G^m(r,t)|
        \left(
        \bar{\ap}_j^m(t)-\bar{\ap}_j^m(s)
        \right)
        \\
        &\leq
        2p_j
        \left\|F_1^{\pd_j}\right\|_\infty
        \|\hat G^m\|_T
        \left(
        \bar{\ap}_j^m(t)-\bar{\ap}_j^m(s)
        \right),
    \end{align*}
    where
    \[
        \|\hat G^m\|_T
        :=
        \sup_{u\in[0,T]}|\hat G^m(u)|.
    \]

    Substituting
    \eqref{eq: arrival strip oscillation} and
    \eqref{eq: initial strip oscillation} into
    \eqref{eq: mass oscillation}, we obtain a
    modulus-of-continuity bound for
    $\wmass(\hat{\boldsymbol{\tm}}^m(\cdot))$ in terms of the
    moduli of $U_1^{j,m}$, $\hat G^m$, $\bar G^m$, $\bar G$,
    and $\bar{\ap}_j^m$.
    Because the initial condition is tight, and thus compactly contained, this completes the proof.
    \end{proof}
   \begin{lem}
\label{componenttightnesslem}
    Let
    $\{\hat{\boldsymbol{\ssp}}^m(\cdot)\}_{m=1}^{\infty}$
    be a sequence of diffusion-scaled models satisfying
    Assumption~\ref{assumptions}, and let
    $f\in\mathscr{S}(\R_+^2)$.
    Suppose that, for every $j\in\J$,
    \[
        \left\{
        \left(
        U_f^{j,m}(\cdot),
        \hat G^m(\cdot))
        \right)
        \right\}_{m=1}^{\infty}
    \]
    is C-tight. Then
    \[
        \left\{
        \left(
        \left\langle
        f,\hat{\ssp}_j^m(\cdot)
        \right\rangle
        \right)_{j\in\J}
        \right\}_{m=1}^{\infty}
    \]
    is C-tight in $D([0,\infty),\R^J)$.
\end{lem}

\begin{proof}
    Fix $T>0$. Let $h_j$ and $g_j$ denote the densities of
    $\fl_{0,j}$ and $\pd_j$, respectively. For $t\geq0$, define
    \begin{align*}
        \mathcal I_{f,j}^m(t)
        &:=
        \sqrt{m}
        \left\langle
        t_{p_j\bar G^m(0,t),t}f
        -
        t_{p_j\bar G(0,t),t}f,
        \fl_{0,j}
        \right\rangle,
    \end{align*}
    and, for $0\leq r\leq t$, define
    \begin{align*}
        \mathcal R_{f,j}^m(r,t)
        &:=
        \sqrt{m}
        \left\langle
        t_{p_j\bar G^m(r,t),t-r}f
        -
        t_{p_j\bar G(r,t),t-r}f,
        \pd_j
        \right\rangle.
    \end{align*}
    Applying \eqref{mainMeq} and \eqref{Udef}, we obtain
    \begin{align}
        \left\langle
        f,\hat{\ssp}_j^m(t)
        \right\rangle
        &=
        U_f^{j,m}(t)
        +
        \mathcal I_{f,j}^m(t)
        +
        \int_0^t
        \mathcal R_{f,j}^m(r,t)
        \,d\bar{\ap}_j^m(r).
        \label{eq: general f representation}
    \end{align}

    We first derive an oscillation bound for
    $\mathcal R_{f,j}^m$. Since $\pd_j$ has density $g_j$,
    \begin{align*}
        \mathcal R_{f,j}^m(r,t)
        &=
        \sqrt{m}
        \int_{\R_+^2}
        f(x,y)
        \Big[
        g_j\bigl(
        x+p_j\bar G^m(r,t),y+t-r
        \bigr)
        \\
        &\hspace{43mm}
        -
        g_j\bigl(
        x+p_j\bar G(r,t),y+t-r
        \bigr)
        \Big]
        \,d\lambda(x,y).
    \end{align*}
Applying the mean value theorem, we obtain
    \begin{align*}
        &
        \Big|
        \bigl[
        g_j(x+p_j\bar{G}^m(r,t),y+t-r)-g_j(x+p_j\bar{G}(r,t),y+t-r)
        \bigr]
        \\&-
        \bigl[
        g_j(x+p_j\bar{G}^m(r,s),y+s-r)-g_j(x+p_j\bar{G}(r,s),y+s-r)
        \bigr]
        \Big|
        \\
        &\leq
        \frac{p_j}{\sqrt{m}}\|g_{j,x}\|_\infty
        |\hat{G}^m(s,t)|
        \\
        &\quad+
        \frac{p_j}{\sqrt{m}}|\hat{G}^m(r,s)|
        \left(
        \|g_{j,xx}\|_\infty p_j\bar{G}(s,t)
        +
    \|g_{j,xy}\|_\infty (t-s)
        \right).
    \end{align*}
    It then follows that
    \begin{align}
        &
        \left|
        \mathcal R_{f,j}^m(r,t)
        -
        \mathcal R_{f,j}^m(r,s)
        \right|
        \notag\\
        &\leq
        p_j
        \|f\|_{L^1(\lambda)}
        \|g_{j,x}\|_\infty
        |\hat G^m(s,t)|
        \notag\\
        &\quad+
        p_j
        \|f\|_{L^1(\lambda)}
        |\hat G^m(r,s)|
        \left[
        p_j\|g_{j,xx}\|_\infty
        \bar G(s,t)
        +
        \|g_{j,xy}\|_\infty(t-s)
        \right].
        \label{eq: general f arrival oscillation}
    \end{align}

    The same calculation, with $h_j$ in place of $g_j$ and
    $r=0$, gives
    \begin{align}
        &
        \left|
        \mathcal I_{f,j}^m(t)
        -
        \mathcal I_{f,j}^m(s)
        \right|
        \notag\\
        &\leq
        p_j
        \|f\|_{L^1(\lambda)}
        \|h_{j,x}\|_\infty
        |\hat G^m(s,t)|
        \notag\\
        &\quad+
        p_j
        \|f\|_{L^1(\lambda)}
        |\hat G^m(0,s)|
        \left[
        p_j\|h_{j,xx}\|_\infty
        \bar G(s,t)
        +
        \|h_{j,xy}\|_\infty(t-s)
        \right].
        \label{eq: general f initial oscillation}
    \end{align}
By the mean value theorem,
    \[
        |\mathcal R_{f,j}^m(r,t)|
        \leq
        p_j
        \|f\|_{L^1(\lambda)}
        \|g_{j,x}\|_\infty
        |\hat G^m(r,t)|.
    \]
    Therefore,
    \begin{align}
        \int_s^t
        |\mathcal R_{f,j}^m(r,t)|
        \,d\bar{\ap}_j^m(r)
        &\leq
        2p_j
        \|f\|_{L^1(\lambda)}
        \|g_{j,x}\|_\infty
        \|\hat G^m\|_T
        \notag\\
        &\qquad\times
        \bigl(
        \bar{\ap}_j^m(t)-\bar{\ap}_j^m(s)
        \bigr),
        \label{eq: general f new arrivals}
    \end{align}
    where
    \[
        \|\hat G^m\|_T
        :=
        \sup_{u\in[0,T]}|\hat G^m(u)|.
    \]
    Applying \eqref{eq: general f representation} at times
    $s$ and $t$, we obtain
    \begin{align}
        &
        \left|
        \left\langle
        f,\hat{\ssp}_j^m(t)
        \right\rangle
        -
        \left\langle
        f,\hat{\ssp}_j^m(s)
        \right\rangle
        \right|
        \notag\\
        &\leq
        |U_f^{j,m}(t)-U_f^{j,m}(s)|
        +
        |\mathcal I_{f,j}^m(t)-\mathcal I_{f,j}^m(s)|
        \notag\\
        &\quad+
        \int_0^s
        |\mathcal R_{f,j}^m(r,t)
        -
        \mathcal R_{f,j}^m(r,s)|
        \,d\bar{\ap}_j^m(r)
        \notag\\
        &\quad+
        \int_s^t
        |\mathcal R_{f,j}^m(r,t)|
        \,d\bar{\ap}_j^m(r).
        \label{eq: general f total oscillation}
    \end{align}
Controlled oscillations then follow from C-tightness of $U^{j,m}_f$ and $\hat{G}^m,$ as desired.  Because the initial condition is tight, and thus compactly contained, this completes the proof.
\end{proof}

In order to obtain convergence of $U_f^j(\cdot)$ for $f \in \mathscr{S}\cup \{1\},$ we first prove convergence of $\hat{\mart}^{\ap_j,j,m}_{t_{p_j\bar{G}^m(\cdot,t),t-\cdot}f^t}(t)$ for $f \in \mathscr{S}\cup \{1\}.$
\begin{proof}[Proof of Theorem \ref{Ydefthm}]
The proof is the same as the proof of Theorem 3.1 in \cite{loeser2025diffusionlimitsmeasurevaluedqueueing}, and is thus omitted.
\end{proof}
\begin{lem}
    For each finite collection of functions \(f_1,\ldots,f_n\in\mathscr{S}(\R_+^2)\cup \{1\}\),
    \[
\left\{\Phi_{f_a,t}^{j,m}(r)
,
a\in[n],\ j\in[J],\ 0\le r\le t\le T
\right\}
\Rightarrow \left\{
\hat{\mart}_{f_a}^{\ap_j,j}(r,t),
a\in[n],\ j\in[J],\ 0\le r\le t\le T
\right\}
\]
as a multiparameter process in the Skorokhod space, where the convergence is in distribution.
    \normalsize
    \label{martingalelimlem}
\end{lem}
\begin{proof}
For $1 \leq a \leq n,$ $j \in \J,$ $r, s, t\in [0,T]$ define
\begin{align*}
    &X^{\ap,j,m}_{f_a}(r,s,t)\\&:=\frac{1}{\sqrt{m}}\sum_{i=1}^{\lfloor mr\rfloor}\left(t_{(s-p_j\bar{G}^m(\iinta_i^{j,m}/m))^+, (t-\iinta_i^{j,m}/m)^+}f_a(\ser_i^j, \pat_i^j)- \langle t_{(s-p_j\bar{G}^m(\iinta_i^{j,m}/m))^+, (t-\iinta_i^{j,m}/m)^+}f_a,\pd_j\rangle\right), 
\end{align*}
and note that, applying \eqref{eq: diffusion scaled Y}, $\Phi_{f_a,t}^{j,m}(r)=X^{\ap,j,m}_{f_a}(\bar{\ap}^m_j(r),p_j\bar{G}^m(t),t).$
    Applying Corollary 4.1 of \cite{loeser2025diffusionlimitsmeasurevaluedqueueing} we see that, fixing $(s_1,t_1),...,(s_k,t_k),$
    \begin{align*}
       & \left\{
X^{\ap,j,m}_{f_a}(r,s_i,t_i),
a\in[n],\ j\in[J], i \in [k],  r\in [0,T]
\right\}\\
&\Rightarrow \left\{
\hat{\mart}_{f_a}^{\ap_j,j}((\ar_j^{-1}r)\wedge t_i,\bar{G}^{-1}(s_i/p_j),t_i),
a\in[n],\ j\in[J], i\in [k],r \in [0,T]
\right\}.
\numberthis
\label{eq: fdd for X}
    \end{align*}
 
We remark that this is a reasonably straightforward application of the martingale central limit theorem (see e.g., \cite{ethierandkurtz}, Chapter 7, Theorem 1.4 (b)).
However, should the reader want to see all of the assumptions checked and the quadratic covariations explicitly calculated for martingales of this form, all of the details are available in the proof of Corollary 4.1 from \cite{loeser2025diffusionlimitsmeasurevaluedqueueing}.
We claim that, if $\{X^{\ap,j,m}_{f_a}(\cdot,\cdot,\cdot)\}_{m=1}^{\infty}$ is C-tight, the desired convergence follows.
Indeed, if $\{X^{\ap,j,m}_{f_a}(\cdot,\cdot,\cdot)\}_{m=1}^{\infty}$ is C-tight as a multiparameter process, we find that it converges in distribution to the unique process with finite-dimensional distributions given by \eqref{eq: fdd for X}.
Furthermore, because $\bar{G}^m(\cdot),\bar{\ap}_j^m(\cdot)$ are time-changes with continuous deterministic limits, C-tightness of $\{\Phi_{f_a,t}^{j,m}\}_{m=1}^{\infty}$ follows from C-tightness of $\{X^{\ap,j,m}_{f_a}(\cdot,\cdot,\cdot)\}_{m=1}^{\infty}.$
Finally, again using the fact that $\bar{G}^m(\cdot),\bar{\ap}_j^m(\cdot)$ are time-changes with continuous deterministic limits, we notice that any subsequential limit of $\{\Phi_{f_a,t}^{j,m}\}_{m=1}^{\infty}$ would have f.d.d. equal to those of the process $\hat{\mart}_{f_a}^{\ap_j,j}$.
This reduces the proof to C-tightness of $X^{\ap,j,m}_{f_a}.$

Therefore, all that remains to be checked is C-tightness of  the process $X^{\ap,j,m}_{f_a}$.
The proof of C-tightness is very similar to the proof of Lemma 6.4 of \cite{loeser2025diffusionlimitsmeasurevaluedqueueing}, particularly for the $\hat{Y}^{\ap_j,j}_f(\cdot)$ families of martingales.
However, because we have moved from measures on $\R_+$ to measures on $\R_+^2,$ the proof will need to be slightly adjusted.
In particular, one must replace the time changed process $X^{\ap,m}(r,t)$ with
$X^{\ap,m}(r,s,t)$.
Then, defining the $(s_1,s_2]\times (t_1,t_2]$ increment function of a Borel measurable function $f:\R_+^2\rightarrow \R$ as follows:
$$\Delta_{(s_1,s_2]\times (t_1,t_2]}f:=t_{s_1,t_1}f-t_{s_2,t_1}f-t_{s_1,t_2}f+t_{s_2,t_2}f,$$
one may replace the incremental terms $\xi_i(t_1,t_2)$ given in the proof from \cite{loeser2025diffusionlimitsmeasurevaluedqueueing} with the incremental terms
\begin{align*}
    \xi_i^{\ap}((s_1,s_2]\times(t_1,t_2])&:= \Delta_{((s_1-p_j\bar{G}^m(\iinta_i^{j,m}/m))^+,(s_2-p_j\bar{G}^m(\iinta_i^{j,m}/m))^+]\times((t_1-\iinta_i^{j,m}/m)^+,(t_2-\iinta_i^{j,m}/m)^+]}f(\ser_i^j,\pat_i^j)\\&-\langle \Delta_{((s_1-p_j\bar{G}^m(\iinta_i^{j,m}/m))^+,(s_2-p_j\bar{G}^m(\iinta_i^{j,m}/m))^+]\times((t_1-\iinta_i^{j,m}/m)^+,(t_2-\iinta_i^{j,m}/m)^+]}f,\pd_j\rangle.
\end{align*}
We also define an inverse increment of a Borel measurable function $f:\R_+^2\rightarrow \R$ as
$$\Delta_{(s_1,s_2]\times (t_1,t_2]}^{-1}f(x,y):= f(x+s_1,y+t_1)-f(x+s_2,y+t_1)-f(x+s_1,y+t_2)+f(x+s_2,y+t_2).$$
Then we observe that
\begin{align*}
     &E[(\xi_i^{\ap}((s_1,s_2]\times(t_1,t_2]))^2]\\&= E[E[(\xi_i^{\ap}((s_1,s_2]\times(t_1,t_2]))^2|\bar{G}^m(\iinta_i^{j,m}/m),\iinta_i^{j,m}]]\\
     & \leq |E[\langle (\Delta_{((s_1-p_j\bar{G}^m(\iinta_i^{j,m}/m))^+,(s_2-p_j\bar{G}^m(\iinta_i^{j,m}/m))^+]\times((t_1-\iinta_i^{j,m}/m)^+,(t_2-\iinta_i^{j,m}/m)^+]}f)^2,\pd_j\rangle]|\numberthis \label{firstsecondmomentbound}\\&+|E[\langle \Delta_{((s_1-p_j\bar{G}^m(\iinta_i^{j,m}/m))^+,(s_2-p_j\bar{G}^m(\iinta_i^{j,m}/m))^+]\times((t_1-\iinta_i^{j,m}/m)^+,(t_2-\iinta_i^{j,m}/m)^+]}f,\pd_j\rangle^2]| \numberthis \label{secondsecondmomentbound}
\end{align*}
Next, we observe that, by the fundamental theorem of calculus, for any $L^1$ function $f:\R_+^2\rightarrow \R$ and $0 \leq s_1 \leq s_2,0 \leq t_1 \leq t_2,$
$$|\langle \Delta_{(s_1,s_2]\times (t_1,t_2]}f, \pd_j\rangle| \leq \|f\|_{L^1}\|g_{st}^j\|_\infty(s_2-s_1)(t_2-t_1) $$
and by direct computation, for $0 \leq s_1 \leq s_2,0 \leq t_1 \leq t_2,$
$$\langle 1_{(s_1-p_j\bar{G}^m(\iinta_i^{j,m}/m),s_2-p_j\bar{G}^m(\iinta_i^{j,m}/m)]\times(t_1-\iinta_i^{j,m}/m,t_2-\iinta_i^{j,m}/m]},\pd_j\rangle \leq \|g^j\|_\infty(s_2-s_1)(t_2-t_1).$$
Therefore, taking $C=(\|f\|_{L^1} \|g^{j}_{st}\|_\infty\vee \|g^j\|_\infty)^2, $ we obtain the desired moment bound of
$ C(s_2-s_1)(t_2-t_1)$ for \eqref{secondsecondmomentbound} so long as $|t_2-t_1| \vee |s_2-s_1| \leq 1.$
Now, we examine
$|\langle (\Delta_{(s_1,s_2]\times (t_1,t_2]}f)^2, \pd_j\rangle|$ for $0 \leq s_1 \leq s_2, 0 \leq t_1 \leq t_2.$ For the $f\equiv 1$ case, it is straightforward to verify that $\left(\Delta_{(s_1,s_2]\times(t_1,t_2]}1\right)^2
=
\Delta_{(s_1,s_2]\times(t_1,t_2]}1$ for $0\le s_1\le s_2$ and $0\le t_1\le t_2$, so we may use the bound for the right-hand side.
For $f \in L^1(\R_+^2),$ the bound is significantly more technical.
We expand
\begin{align*}
    \langle (\Delta_{(s_1,s_2]\times (t_1,t_2]}f)^2, \pd_j\rangle&= \langle t_{s_1,t_1}f\Delta_{(s_1,s_2]\times (t_1,t_2]}f, \pd_j \rangle \numberthis \label{squarebitone}\\
    & - \langle t_{s_1,t_2}f\Delta_{(s_1,s_2]\times (t_1,t_2]}f, \pd_j \rangle \numberthis \label{squarebittwo}\\
    & - \langle t_{s_2,t_1}f\Delta_{(s_1,s_2]\times (t_1,t_2]}f, \pd_j \rangle\numberthis \label{squarebitthree}\\
    &+\langle t_{s_2,t_2}f\Delta_{(s_1,s_2]\times (t_1,t_2]}f, \pd_j \rangle \numberthis \label{squarebitfour}
\end{align*}
For \eqref{squarebitone}, we see that
\begin{align*}
    |\langle t_{s_1,t_1}f\Delta_{(s_1,s_2]\times (t_1,t_2]}f, \pd_j \rangle|
    &= \bigg|\int_0^{\infty}\int_0^{\infty}
    f(x,y)\Delta^{-1}_{[0,s_2-s_1)\times [0,t_2-t_1)}
    \left(f(x,y)g^j(x+s_1,y+t_1)\right)\,dx\,dy\bigg|\\
    & \leq \|f\|_{L^{1}} C_1(t_2-t_1)(s_2-s_1)
\end{align*}
where $C_1$ depends only on the $L^{\infty}$ norms of the first and second derivatives of $f,g^j.$
Similarly, for \eqref{squarebitfour}
\begin{align*}
   | \langle t_{s_2,t_2}f\Delta_{(s_1,s_2]\times (t_1,t_2]}f, \pd_j \rangle |
   &=\bigg|\int_0^{\infty}\int_0^{\infty}
   f(x,y)g^j(x+s_2,y+t_2)
   \Delta^{-1}_{[0,s_2-s_1)\times [0,t_2-t_1)}f(x,y)\,dx\,dy\bigg|\\
    & \leq \|f\|_{L^1} \|g^j\|_\infty C_2(t_2-t_1)(s_2-s_1)
\end{align*}
where $C_2$ only depends on the $L^{\infty}$ norms of the first and second derivatives of $f.$
For \eqref{squarebittwo}, we have
\begin{align*}
    |\langle t_{s_1,t_2}\Delta_{(s_1,s_2]\times (t_1,t_2]}f, \pd_j \rangle| & \leq \bigg| \int_0^{\infty}\int_0^{s_2-s_1} f(x,y)(f(x,y)-f(x,y+t_2-t_1))g^j(x+s_1,y+t_2)\,dx\,dy\bigg|\\
    &+\bigg|\int_0^{\infty}\int_0^{\infty} f(x+s_2-s_1,y)( \Delta^{-1}_{(0,s_2-s_1]\times(0, t_2-t_1]}f(x,y))g^j(x+s_2,y+t_2)\,dx\,dy\bigg|\\
    & \leq C_3( L_{\pd_j^x} \|f\|_\infty + \|f\|_{L^1}\|g^j\|_\infty)(t_2-t_1)(s_2-s_1)
\end{align*}
where $C_3$ is again a constant that only depends on the $L^{\infty}$ norms of the first and second derivatives of $f.$
The same principle applies to \eqref{squarebitthree}.
We conclude that there exists some $C$ such that
$$E[(\xi_i^{\ap}((s_1,s_2]\times(t_1,t_2]))^2] \leq C(t_2-t_1)(s_2-s_1).$$
All other elements of the proof remain exactly the same.
\end{proof}
\begin{lem}

\label{uconvlemma}
Let $f_1,...,f_n \in {\mathscr{S}(\R_+^2)}\cup \{1\}.$ 
Then, the sequence
\[
\boldsymbol{U}_{\boldsymbol{f}}^{m}(\cdot)
:=
\left(
U_{f_a}^{j,m}(\cdot):
a\in[n],\ j\in[J]
\right)
\]
converges in distribution to the process $\boldsymbol{U}_{\boldsymbol{f}}(\cdot)$ defined in Definition \ref{Udefdef}.
\end{lem}
\begin{proof}
For each $f \in \mathscr{S}(\R_+^2)\cup \{1\},$ define the multi-index process
 \begin{align*}
        U_{f}^{j,m}(r,t) &= F_{f}^{j,c,m}(p_j\bar{G}^m(t),t) +\Phi_{f,t}^{j,m}(r) 
  +\int_0^r\langle t_{p_j\bar{G}(s,t),t-s}f, \pd_j \rangle d\hat{\ap}_j^m(s).
     \numberthis \label{Umultiindexdef}
    \end{align*}
Then, fixing $t_1,...,t_l\geq 0,$ one may apply the Central Limit Theorem for Renewal-Driven Processes proved in \cite{loeser2025diffusionlimitsmeasurevaluedqueueing} (Theorem 5.1) to see that for $f_1,...,f_n \in \mathscr{S}(\R_+^2) \cup \{1\},$
\[
(\boldsymbol{U}_{\boldsymbol{f}}^{m}(\cdot,t_1),...,\boldsymbol{U}_{\boldsymbol{f}}^{m}(\cdot,t_l))
:=
\left(
U_{f_a}^{j,m}(\cdot,t_i):
a\in[n],\ j\in[J], i \in [l]
\right)
\]
converges to a process whose f.d.d. are equal to those of the process defined in Definition \ref{Udefdef}.
Therefore, as long as $U_f^{j,m}(\cdot,\cdot)$ is C-tight as a multi-index process, the result follows from restriction to the diagonal. In particular, since every subsequential limit of the multi-index
process is continuous, the diagonal restriction is continuous at the
limiting process, and the continuous mapping theorem applies.
To show this, we simply show that each term on the right-hand side of \eqref{Umultiindexdef} is C-tight.
The first term on the right-hand side converges by Assumption \ref{assumptions} \ref{initialconditionsassumption}.
C-tightness for the second term on the right-hand side is proved in Lemma \ref{martingalelimlem}.
The only term left to check is the last one.
    The proof of convergence of this term is the same as the proof of convergence for the analogous term in \cite{loeser2025diffusionlimitsmeasurevaluedqueueing}, which is given in the proof of Lemma 6.5 of that paper, so it is omitted.
    To briefly summarize, that proof extends Lemmas 4.2 and 4.3 of \cite{loeser2025diffusionlimitsmeasurevaluedqueueing}, which establish stochastic-integral convergence for deterministic integrands against diffusion-scaled renewal processes, to integrands that depend continuously on the upper limit of integration $t$ while having total variation bounded uniformly in $t$.
    To verify the total variation condition, note that, when $f\equiv 1$, the map
\[
s\mapsto
\left\langle
t_{p_j\bar{G}(s,t),t-s}1,\pd_j
\right\rangle
\]
is monotone on $[0,t]$ and takes values in $[0,1]$. Hence
\[
\sup_{t\leq T}
TV\left(
\left\langle
t_{p_j\bar{G}(\cdot,t),t-\cdot}1,\pd_j
\right\rangle
\right)_{[0,t]}
\leq 1.
\]
For $f\in\mathscr{S}(\R_+^2)$, let $g^j$ denote the density of
$\pd_j$. By a change of variables,
\[
\left\langle t_{p_j\bar{G}(s,t),t-s}f,\pd_j\right\rangle
=
\int_{\R_+^2}
f(x,y)
g^j(x+p_j\bar{G}(s,t),y+t-s)
\,dx\,dy.
\]
Thus, for $0\leq s_1\leq s_2\leq t\leq T$, the mean value
theorem gives
\begin{align*}
&\left|
\left\langle t_{p_j\bar{G}(s_2,t),t-s_2}f,\pd_j\right\rangle
-
\left\langle t_{p_j\bar{G}(s_1,t),t-s_1}f,\pd_j\right\rangle
\right|
\\
&\leq
\|f\|_{L^1}
\left(
p_j\|g_x^j\|_\infty
|\bar{G}(s_2,t)-\bar{G}(s_1,t)|
+
\|g_y^j\|_\infty|s_2-s_1|
\right).
\end{align*}
Since $\bar{G}(s,t)$ is Lipschitz in $s$ with Lipschitz constant $C_T$,
\[
\left|
\left\langle t_{p_j\bar{G}(s_2,t),t-s_2}f,\pd_j\right\rangle
-
\left\langle t_{p_j\bar{G}(s_1,t),t-s_1}f,\pd_j\right\rangle
\right|
\leq
(C_T \vee 1) \|f\|_{L^1}
\left(
p_j\|g_x^j\|_\infty+\|g_y^j\|_\infty
\right)
|s_2-s_1|.
\]
Consequently,
\[
\sup_{t\leq T}
TV\left(
\left\langle
t_{p_j\bar{G}(\cdot,t),t-\cdot}f,\pd_j
\right\rangle
\right)_{[0,t]}
\leq
(C_T \vee 1) T\|f\|_{L^1}
\left(
p_j\|g_x^j\|_\infty+\|g_y^j\|_\infty
\right)
<\infty.
\]
\end{proof}
\begin{proof}[Proof of Theorem \ref{tightnessresult}]
Applying Theorem 4.1 of \cite{mitoma} and its extension to all of $\R_+$ in (R.2.2) of the same work, it suffices to check that each projection $\langle f,\hat{\ssp}^m_j(
    \cdot)\rangle $ is tight for all $f \in \mathscr{S}(\R_+^2).$
    Therefore, the tightness result follows immediately from Lemma \ref{componenttightnesslem} along with the C-tightness of $\{(\hat{G}^m(\cdot),{U}_f^{j,m}(\cdot))\}_{m=1}^{\infty}$.
\end{proof}
We now prove the final two main results. 

\begin{proof}[Proof of Theorems \ref{Llimit} and \ref{lhatconvergencethm}]
Let $\wmass(\hat{\boldsymbol{\tm}}(\cdot)),\hat{G}(\cdot))$ be a subsequential limit in distribution of $\{(\wmass(\hat{\boldsymbol{\tm}}^m(\cdot)),\hat{G}^m(\cdot))\}_{m=1}^{\infty}.$
By a slight abuse of notation, we will use $m,$ rather than $m_k,$ to index the converging subsequence.
Fix $f_1,\ldots,f_n\in\mathscr{S}(\R_+^2)$. Passing to a further
subsequence if necessary, we take a Skorokhod representation under
which
\[
\begin{gathered}
\wmass(\hat{\boldsymbol{\tm}}^m(\cdot)),\quad
\hat{G}^m(\cdot),\quad
\wmass(\bar{\boldsymbol{\tm}}^m(\cdot)),\quad
\bar{G}^m(\cdot),\\
\bar{\ap}_j^m(\cdot),\quad
\hat{\ap}_j^m(\cdot),\quad
U_1^{j,m}(\cdot),\quad
U_{f_a}^{j,m}(\cdot),\quad
\langle f_a,\hat{\ssp}_j^m(\cdot)\rangle,
\qquad a\in[n],\ j\in\J,
\end{gathered}
\]
converge jointly almost surely to their corresponding subsequential
limits.
It is immediate from the uniform convergence on compact sets of the integrand combined with bounded convergence that the integral relationship $\hat{G}^m(\cdot) = -\int_0^{\cdot}\frac{\wmass(\hat{\boldsymbol{\tm}}^m(u))}{\wmass(\bar{\boldsymbol{\tm}}^m(u))\wmass(\boldsymbol{\flm}(u))}du$ holds in the limit.
Therefore \eqref{LGsystem1} is proved.

We obtain \eqref{LGsystem2} by taking term-by-term limits of the expressions on the RHS of \eqref{mainMeq} for $f \equiv 1.$ 
Furthermore, it suffices to prove the equality for each fixed $t \geq 0$, as we have already established continuity of the limiting process $\wmass(\hat{\boldsymbol{\tm}}(\cdot))$, and $\R_+$ is a separable index set.
The convergence of $U_f^{j,m}(\cdot)$ is given by Lemma \ref{uconvlemma}.
For the second term on the right hand side of \eqref{mainMeq}, we will apply the mean value theorem to get it into the form of the second term in \eqref{LGsystem2}.
$$\sqrt{m}sgn(\bar{G}(0,t)-\bar{G}^m(0,t)) \int_{p_j\mathcal{G}_-^m(t)}^{p_j\mathcal{G}_-^m(t)+|p_j\hat{G}^m(t)|/\sqrt{m}}\int_t^{\infty}h^j(x,y)dy,$$
where $h^j(x,y)$ is the density of $\fl_j(0),$ which we've taken to be in $C^2_b(\R_+^2)$ and assumed to have finite, continuous marginals.
It follows that, for each $t\geq 0,$ the function
$h^{j,t}(x):= \int_t^{\infty}h^j(x,y)dy,$
is a jointly continuous function in $t,x$.
Therefore, applying the mean value theorem for integrals, the second term on the RHS of \eqref{mainMeq} simplifies to 
$$ -p_j\hat{G}^m(t)h^{j,t}(p_j\mathcal{G}^m_-(t) + \xi^m(t))$$
for some $\xi^m(t) \in [0,p_j|\hat{G}^m(t)|/\sqrt{m}]$ for each $m,t.$ Using the continuity of $h^{j,t},$ we obtain the limit $-p_j\hat{G}(t)\int_t^{\infty} h^j(p_j\bar{G}(t),y)dy.$

For the last term on the RHS of \eqref{mainMeq}, doing the same calculation, but with $g^{j,t}(x):=\int_t^{\infty}g^j(x,y)dy,$ we may rewrite the term as
$$-\int_0^tg^{j,t-s}(p_j\mathcal{G}^m_-(s,t)+\Upsilon^m(s,t))  p_j\hat{G}^m(s,t) d\bar{\ap}^m_j(s). $$
Then, applying a standard real analysis argument using the joint continuity of the function $g^{j,t}(x)$ in $t,x$ and the fact that the Lebesgue-Stieltjes measure on $[0,t]$ induced by $\bar{\ap}_j^m(s)$ converges weakly to the measure $\ar_j \lambda,$ where $\lambda$ is the Lebesgue measure on $[0,t],$ we obtain the last term in \eqref{LGsystem2}.
The interested reader may find the details of this real analysis argument in Lemma 9.2.2 of \cite{loeser2024fluid}.

It remains to prove uniqueness in distribution of solutions to the system of equations \eqref{LGsystem1}-\eqref{LGsystem2}.
Let $(\wmass(\hat{\boldsymbol{\tm}}(\cdot)),\hat{G}(\cdot))$, $(\wmass(\tilde{\boldsymbol{\tm}}(\cdot)),\tilde{G}(\cdot))$ be two different solutions to the system for the same realization of $\boldsymbol{U}_1(r,t)$.
Then, using the fact that $\wmass({\boldsymbol{\flm}}(\cdot))$ is uniformly bounded below by some $1/C,$ we obtain from \eqref{LGsystem1}
$$|\hat{G}(\cdot)-\tilde{G}(\cdot)| \leq \int_0^{\cdot} C^2 |\wmass(\hat{\boldsymbol{\tm}}(u))-\wmass(\tilde{\boldsymbol{\tm}}(u))|du.$$
Combining this with \eqref{LGsystem2}, we find that
\begin{align*}
    |\hat{G}(t)-\tilde{G}(t)|&\leq C^2\int_0^t |\wmass(\hat{\boldsymbol{\tm}}(r))-\wmass(\tilde{\boldsymbol{\tm}}(r))|dr\\
    & \leq \sum_{j=1}^J C^2 p_j  \int_0^tp_j\bigg|\hat{G}(r)-\tilde{G}(r)\bigg|h^{j,r}(p_j\bar{G}(r))dr\\
    &+ \sum_{j=1}^J C^2 p_j \int_0^t\int_0^r \ar_j p_j |\tilde{G}(r)-\hat{G}(r)|g^{j,r-s}(p_j\bar{G}(s,r)) dsdr\\
    &+\sum_{j=1}^J C^2 p_j \int_0^t\int_0^r \ar_j p_j |\tilde{G}(s)-\hat{G}(s)|g^{j,r-s}(p_j\bar{G}(s,r)) dsdr\\
\end{align*}
It follows that
$$|\hat{G}(\cdot)-\tilde{G}(\cdot)| \leq \sum_{j=1}^Jp_jC' \int_0^{\cdot} |\hat{G}(u)-\tilde{G}(u)|du,$$
where the constant $C'$ depends only on $\ar_j,T, C$ and the sup norms of the continuous functions $h^{j,t}(x),$ $g^{j,t}(x)$ over the compact region $t \leq T, x \leq \bar{G}(t).$
Uniqueness of solutions for $\hat{G}(\cdot)$ then follows from Gr{\"o}nwall's inequality.
Uniqueness of solutions of $\wmass(\hat{\boldsymbol{\tm}}(\cdot))$ then follows from \eqref{LGsystem2}.
Because all realizations of $\boldsymbol{U}_1(r,t)$ are the same in distribution, this completes the proof.

Lastly, we must prove that limits of $\{\langle f_1, \hat{\boldsymbol{\ssp}}^m(\cdot)\rangle,...,\langle f_n, \hat{\boldsymbol{\ssp}}^m(\cdot)\rangle \}_{m=1}^{\infty}$ satisfy \eqref{rjlimitdef} for $f_1,...,f_n \in \mathscr{S}(\R_+^2)$. Once again,
we will prove convergence of \eqref{mainMeq} term-by-term to \eqref{rjlimitdef}.
Most of the elements of the proof will be the same as in the $f \equiv 1$ case except that the prelimit and limiting expressions for the second and last terms on the RHS of \eqref{mainMeq} will be different. In particular, we rewrite the second term as 
\begin{align*}
    \sqrt{m}\langle t_{p_j\bar{G}^m(t),t}f-t_{p_j\bar{G}(t),t}f, \fl_j(0)\rangle &=\sqrt{m}\int_{\R_+^2} (h^j(x+p_j\bar{G}^m(t),y+t)-h^j(x+p_j\bar{G}(t),y+t))f(x,y) d\lambda\\
    &=\int_{\R_+^2} p_j\hat{G}^m(t)h^j_1(\xi^m(x,t),y+t)f(x,y) d\lambda
\end{align*}
where $\xi^m(x,t) \in [x+p_j\mathcal{G}_-^m(t),x+p_j\mathcal{G}_+^m(t)]$ for each choice of $m,x,t.$
For the last term on the RHS of \eqref{mainMeq}, we obtain
$$\int_0^t \int_{\R_+^2} p_j\hat{G}^m(s,t)g^j_1( \Upsilon^m(x,s,t),y+t-s)f(x,y) d\lambda d\bar{\ap}_j^m(s)$$
where $\Upsilon^m(x,s,t) \in [x+p_j\mathcal{G}_-^m(s,t),x+p_j\mathcal{G}_+^m(s,t)].$
The rest of the proof of convergence is the same as the $f \equiv 1$ case. 
\end{proof}

\noindent {\bf Acknowledgements} The research reported in this paper was supported in part by NSF RTG grant DMS-2134107.
 Lastly, the author used ChatGPT for grammar, copyediting, organization, proof-checking, and phrasing suggestions. Furthermore, author would like to note that, in the process of proofing this paper, the author gave the proof of Lemma \ref{bdeltabound} to ChatGPT Thinking 5-6 to check over for errors. It noticed a mistake in the second half of the proof, and suggested a great idea for fixing it by breaking up the interval $[0,T]$ into intervals $I_k^{\delta},$ $k=0,...,K_{\delta}-1$.
That part of the proof is written by the author of the paper, but credit for the idea goes to ChatGPT.
After this, the author was careful to ask the LLM to help proofread but not provide any solutions to errors found.
Aside from this, no mathematical results or proofs were generated by the tool, and the author reviewed and edited all AI-assisted text.
Any errors in the above work are the responsibility of the author.

\vspace{0.3cm}

	\small{
		\bibliographystyle{amsalpha}
		\bibliography{ppsdiffbib}
	}

\end{document}